\documentclass[
10pt, 
a4paper, 
oneside, 
headinclude,footinclude, 
]{article}

\usepackage[
nochapters, 
pdfspacing, 
dottedtoc 
]{classicthesis} 
\usepackage{amsopn}
\usepackage[T1]{fontenc} 
\usepackage{multirow}
\usepackage[utf8]{inputenc} 
\usepackage{hhline}
\usepackage{amsmath,amsfonts,amssymb,amsthm,epsfig,epstopdf,titling,url,array}
\usepackage{graphicx} 
\usepackage[a4paper,top=5cm,bottom=5cm,left=4cm,right=4cm]{geometry}
\graphicspath{{Figures/}} 
\usepackage{indentfirst}
\usepackage{enumitem} 
\usepackage{savesym}
\usepackage{mathrsfs}
\usepackage{nicefrac}
\usepackage{esint}
\usepackage{subfig} 

\usepackage{amsmath,amssymb,amsthm,amsfonts} 
\usepackage{geometry}
\usepackage{varioref} 

\usepackage{thmtools}
\usepackage{thm-restate}

\usepackage{hyperref}

\theoremstyle{plain}
\newtheorem{teorema}{Theorem}[section]
\newtheorem{proposizione}[teorema]{Proposition}
\newtheorem{lemma}[teorema]{Lemma}
\newtheorem{corollario}[teorema]{Corollary}

\theoremstyle{definition}
\newtheorem{definizione}{Definition}[section]

\theoremstyle{remark}
\newtheorem{osservazione}{Remark}[section]

\newcommand{\N}{\mathbb{N}}
\newcommand{\Z}{\mathbb{Z}}

\newcommand{\R}{\mathbb{R}}

\newcommand{\res}

\DeclareMathOperator*{\lip}{Lip_1}

\DeclareMathOperator*{\cl}{cl}
\DeclareMathOperator*{\diam}{diam}

\hypersetup{
colorlinks=true, breaklinks=true, bookmarks=true,bookmarksnumbered,
urlcolor=webbrown, linkcolor=RoyalBlue, citecolor=webgreen, 
pdftitle={}, 
pdfauthor={\textcopyright}, 
pdfsubject={}, 
pdfkeywords={}, 
pdfcreator={pdfLaTeX}, 
pdfproducer={LaTeX with hyperref and ClassicThesis} 
}

\title{\normalfont\spacedallcaps{Full non-differentiability sets of typical Lipschitz functions}} 

\author{\spacedlowsmallcaps{Andrea Merlo\textsuperscript{*}}} 

\date{} 

\begin{document}


\renewcommand{\sectionmark}[1]{\markright{\spacedlowsmallcaps{#1}}} 
\lehead{\mbox{\llap{\small\thepage\kern1em\color{halfgray} \vline}\color{halfgray}\hspace{0.5em}\rightmark\hfil}} 

\pagestyle{scrheadings} 


\maketitle 

\setcounter{tocdepth}{2} 
\setcounter{secnumdepth}{2}




\paragraph*{Abstract} 
In this paper we prove that the typical Lipschitz function has no directional derivative at any point of a Borel set $E$ if and only if $E$ is contained in a countable union of closed purely unrectifiable sets.

{\let\thefootnote\relax\footnotetext{* \textit{Scuola Normale Superiore, Piazza dei Cavalieri, 7, Pisa, Italy.}\\
\text{ }\textit{ e-mail}: \texttt{andrea.merlo@sns.it}}
}


\section{Introduction}
The characterisation of the non-differentiability sets of real valued Lipschitz functions on the real line goes back to Zahorski who proved in \cite{zahorski} that:
\begin{teorema}
For any $G_{\delta\sigma}$ subset of the real line of Lebesgue measure zero $E$, there exists a Lipschitz function $f:\R\to\R$ which is non-differentiable everywhere on $E$ and differentiable everywhere on $\R\setminus E$.

Viceversa, given a Lipschitz function $f:\R\to \R$, the set of points at which $f$ is non-differentiable is a $G_{\delta \sigma}$ Lebesgue-null set.
\label{zar}
\end{teorema}

At this point is quite natural to ask whether any similar characterisation is available for Lipschtiz maps $f:\R^n\to\R^m$. As it turns out already when the domain is $\R^2$, the answer is much more complicated and only partially known. Indeed M. Doré and O. Maleva in \cite{MalevaDore} and \cite{MalevaDore1} constructed a compact set with Hausdorff dimension $1$ in the $\R^n$ on which every Lipschitz function has a differentiability point (the first Lebesgue-null set with this property was first constructed by D. Preiss in \cite{Preiss1}).

In order to solve the problem, one could hope to use the intuitive idea that the typical Lipschitz functions have the worst differentiability behaviour, and thus the problem may be solved by means of the Baire Cathegory Theorem on a suitable space of Lipschitz functions.
This approach was attempted in 1995 by D. Preiss and J. Ti\u{s}er in \cite{PreissTiser} where they showed that:

\begin{teorema}[Preiss, Ti\u{s}er]
Let $E\Subset (0,1)$ be an analytic set. The following are equivalent:
\begin{itemize}
\item[(i)]$E$ is contained in an $F_\sigma$ subset of $[0,1]$ with Lebesgue measure zero.
\item[(ii)]The set $S$ of those $1$-Lipschitz functions differentiable at no point of $E$ is residual in $\lip([0,1],\R)$, the space of $1$-Lipschitz functions on $[0,1]$ with values in $\R$ endowed with the supremum norm.
\end{itemize}\label{isp}
\end{teorema}
Theorem \ref{isp} is both good and bad news. On the one hand it shows that if $E$ is covered by countably many closed Lebesgue-null sets, then the Baire Cathegory Theorem produces non-differentiable functions on $E$. On the other, if $E$ does not satisfy this topological condition (for istance if $E\subseteq [0,1]$ is residual and Lebesgue-null), they proved that the typical Lipschitz function \emph{has} a point of differentiability in $E$, showing that this topological approach cannot tell the full story, in view of Zahorski's theorem.

There are many possible generalisations of the above result in higher dimensions. The one in which we are interested is the following: is it possible to build a map from $\R^n$ to $\R^n$ which is non-differentiable in any direction of a given purely unrectifiable Borel set $E$ by means of the Baire Category Theorem?
The answer, which is the main result of this paper, depends on the topological properties of the set $E$ too:

\begin{teorema}\label{prinprin}
Let $E\Subset (0,1)^n$ be an analytic set and let $n\leq m$. Then the follwing are equivalent:
\begin{itemize}
\item[(i)]$E$ is contained in a countable union of closed purely unrectifiable sets,
\item[(ii)] the set $S$ of those $1$-Lipschitz functions which are non-differentiable in every direction at every point of $E$ is residual in $(\lip([0,1]^n,\R^m),\lVert\cdot\rVert_\infty)$, the space of $1$-Lipschitz functions on $[0,1]^n$ with values in $\R^m$ endowed with the supremum norm.
\end{itemize}
\end{teorema}

As in the one dimensional case, the proof of (ii)$\Rightarrow$(i) shows that if (i) does not hold, then the typical Lipschitz function \emph{has} a differentiability point in $E$. This gives an intuitive justification to why even the construction of fully non-differentiable functions on non-compact purely unrectifiable sets in \cite{ACP0.2} is so intricate.

\subsection*{Scheme of the proof}
The proof of the implication $(i)\Rightarrow (ii)$ of Theorem \ref{prinprin} heavily relies on techniques introduced in \cite{AlbertiMarchese} and in \cite{ACP0.2}. Fix $\epsilon>0$, $e\in\mathbb{S}^{2n-1}$ and a closed purely unrectifiable set $E$. It is possible (see Lemma 4.12  in \cite{AlbertiMarchese}) to build for any $e\in\mathbb{S}^{n-1}$ a $1$-Lipschitz function $g_e$ such that:
\begin{itemize}
    \item[($\alpha$)] $\lVert g_e\rVert_\infty\leq \epsilon$,
    \item[($\beta$)] $\lvert Dg_e(x)-e\rvert<\epsilon$ for any $x\in E$.
\end{itemize}
Using these functions it is not hard to construct maps $G:\R^n\to\R^n$ with small supremum norms and such that $DG(x)\approx \text{id}_n$ for any $x\in E$. Pick any smooth function $f$, and let $u\in\mathbb{S}^{n-1}$,  $v\in\mathbb{S}^{m-1}$ and define:
$$\tilde{f}(x):=f(x)-D f(x)[G(x)]+g_u(x)v.$$
The function $\tilde{f}$ is close to $f$ in the supremum topology, and on $E$ its derivative along the direction $v$ is near to $u$. In the end this construction (with some fine tuning) and the density of smooth functions in $\lip([0,1]^n,\R^m)$ prove the implication (i)$\Rightarrow$(ii).

To explain the proof of the implication (ii)$\Rightarrow $(i) we need to introduce the Banach-Mazur game first. Let $E$ be a set which cannot be covered by countably many compact purely unrectifiable sets and define the family of $1$-Lipschitz functions:
\begin{equation}
    \begin{split}
        B:=\{f\in\lip([0,1]^n,\R^m):\text{ there are }&x\in E,~ e\in\R^n\text{ s.t. }
        f(x+te)\text{ is differentiable at }t=0\}.
        \nonumber
    \end{split}
\end{equation}
Consider the following game with two players. Player (I) chooses an open set $U_1\subseteq \lip([0,1]^n,\R^m)$; then Player (II) chooses an open set $V_1\subseteq U_1$; then Player (I) chooses an open set $U_2\subseteq V_1$ and so on. Player (II) wins if $\bigcap_i V_i\subseteq B$, otherwise Player I wins. 
If we can build a winning strategy for Player (II), Theorem \ref{BMgioco} implies that the set $B$ is residual in $\lip([0,1]^n,\R^m)$. 

 The proof that $V_k$ can be chosen in such a fashion that Player (II) wins is based on the following two observations: 
\begin{itemize}
    \item[($\alpha^\prime$)] Theorem \ref{tSoleki} says that $E$ is residual in a closed set $F$ having any portion of positive width (see Definition \ref{defiwi}),
    \item[($\beta^\prime$)] if two continuous piece-wise congruent mappings (see Definition \ref{pwcongr}) are close in the supremum norm, then the set where their directional derivative along $e\in\mathbb{S}^{n-1}$ are not close, has small width with respect to the cone of axis $e$ (of a suitable amplitude).
\end{itemize}
Player II at each turn chooses piece-wise congruent mappings $f_k$, sets $M_k$ and directions $e_k$ (converging to some $e$) such that the sets $V_k$ are (small enough) balls centred at $f_k$. The turn of Player II starts by  arbitrarily picking a piece-wise congruent mapping $f_k$ in $U_k$. The direction $e_k$ is chosen close to $e_{k-1}$ in such a way that the width of $M_{k-1}$ along a cone of axis $e_{k}$ (of sufficiently small amplitude) is positive.
Eventually Player II must deal with the construction of $M_k$. Since $E\cap F$ is residual in $F$, we can find a sequence of relatively open sets $E_k$ in $F$ such that $E\cap F\subseteq \bigcap E_k$.  Let $G_k$ be the set given by point ($\beta^\prime$) which enjoys the two following properties:
\begin{itemize}
    \item[($\alpha^{\prime\prime}$)] the complement of $G_k$ has a complement with very small width,
    \item[($\beta^{\prime\prime}$)] on $G_k$ the function $f_k$ and $f_{k-1}$ have close derivatives along the direction $e_k$.
\end{itemize}
Player II defines $M_k$ to be a non-empty relatively open set in $F$, compactly contained in $E_{k-1}\cap M_{k-1}\cap G_k$. Moreover point ($\alpha^\prime$) insures that we can always find such an $M_k$ having positive width with respect to a cone with axis $e_k$. 

The functions $f_k$ are uniformly converging to some $f\in\lip([0,1]^n,\R^m)$ and the sets $M_k$ are constructed in such a way that their intersection is non-empty (thanks to the finite intersection property of compact sets), it is contained in $E\cap F$ and point ($\beta^{\prime\prime}$) implies that $f$ is differentiable along $e$ at any point of $\bigcap M_k$.

\subsection*{Related results}
The problem of the characterisation of non-differentiability sets of Lipschitz functions between Euclidean spaces has quite a long history, originally motivated by the attempt to prove a Rademacher-type theorems on Banach spaces (see for istance the monograph \cite{PLT}). The paper \cite{Preiss1} by D. Preiss could be arguably considered the first fundamental contribution to the theory, where among other things, he constructs a Lebesgue-null set in $\R^2$ on which every Lipschitz function has a differentiability point, showing that Rademacher's Theorem does not tell the full story.
In 2005 G. Alberti, M. Cs\"ornyei and D. Preiss announced in \cite{ACP0} and \cite{ACP0.1} a geometric characterisation of non-differentiability sets of Lipschitz functions and the proof that any Lebesgue-null set in $\R^2$ is contained in a non-differentiability set of some Lipschitz function $f:\R^2\to\R^2$. On the other hand, more recently D. Preiss and G. Speight proved in \cite{PreissSpeight} that for any $m< n$ there exists a Lebesgue-null set $\mathcal{N}\subseteq \R^n$ for which every Lipschitz map $f:\R^n\to \R^m$ has a point of differentiability on $\mathcal{N}$.

On the measure-theoretic side, in 2015 G. Alberti and A. Marchese proved in \cite{AlbertiMarchese} that the Rademacher Theorem can be extended to finite mass Borel measures (when the definition of differentiability is suitably weakened) and in 2016 G. De Philippis and F. Rindler showed in \cite{Afree} that if every Lipschitz function is differentiable $\mu$-a.e. in the standard sense then $\mu$ is absolutely continuous with respect to Lebesgue.

\subsection*{Structure of the paper}
In Section \ref{prelprel} we briefly give the definition of width and some of its properties, while Section \ref{S2} is devoted to state the main Theorem \ref{main} and reduce its proof to Proposition \ref{evev} and Proposition \ref{BMstrategia}. The two remaining sections will deal with the proof of these two propositions: the entire Section \ref{albalb} is devoted to the proof of Proposition \ref{evev} and in Section \ref{Ss4} is contained the proof of Proposition \ref{BMstrategia}.

\section*{Acknowledgements}
A first version of this work was written during my last undergraduate year which I spent at the University of Warwick with the support of the Erasmus+ program and the Scuola Galileiana di Studi Superiori in Padua. I am deeply grateful to David Preiss for suggesting me the problem and for his kind guidance.

\section*{Notation}
We add below a list of frequently used notations:
\medskip

\begin{tabular}{p{3cm} p{0.7\textwidth}}
$\lvert\cdot\rvert$ & Euclidean norm,\\
$\langle\cdot,\cdot\rangle$ & scalar product,\\
$\lVert\cdot\rVert$ & operatorial norm of matrices,\\
$\lVert\cdot\rVert_\infty$ & supremum norm of functions restricted to $[0,1]^n$,\\
$\mathbb{S}^n$ & unit sphere in $\R^{n+1}$,\\
$B_r(x)$ &  open ball of radius $r>0$ and centre x,\\
$B_\delta(A)$ & open neighbourhood of radius $\delta>0$ of the set $A$,\\
$\lip([0,1]^n,\R^m)$ & $1$-Lipschitz functions from $[0,1]^n$ to $\R^m$ (see Definition \ref{deffi}),\\
$\mathfrak{P}(n,m)$ & piece-wise congruent mappings (see Definition \ref{pwcongr}),\\
$M(n,m)$ & $n\times m$ matrices,\\
$O(n,m)$ & matrices representing linear isometries from $\R^n$ to $\R^m$,\\
$\mathcal{P}([0,1]^n)$ & power set of $[0,1]^n$,\\
$e^\perp$ & orthogonal hyperplane to the vector $e$,\\
$C(e,\sigma)$ & proper cone of amplitude $\sigma$ and axis $e$,\\
$\mathcal{L}^n$ & $n$-dimensional Lebesgue measure,\\ 
$\mathcal{H}^1$ & $1$-dimensional Hausdorff meausure.
\end{tabular}

\medskip

A set in a complete metric space is said to be \emph{analytic} if it is a continuos image of a complete metric space.
It is well known that every Borel set in $\R^n$ is analytic.

Finally we briefly recall some standard terminology used to denote the Baire category of sets throughout the paper. Let $(X,\mathcal{T})$ be a topological space and suppose $A\subseteq X$:
\begin{itemize}
\item[(i)] if $\text{int}(\text{cl}(A))=\emptyset$, $A$ is said to be \emph{nowhere dense},
\item[(ii)] if $A$ is the countable union of nowhere dense sets, $A$ is said to be \emph{meagre},
\item[(iii)] if $A$ is the complement of a meagre set, $A$ is said to be \emph{residual}. 
\end{itemize}

\section{Preliminary results}
\label{prelprel}

In this first section we recall used facts on purely unrectifiable sets and introduce the space of piece-wise congruent mappings $\mathfrak{P}(n,m)$, proving their density in $\lip([0,1]^n,\R^m)$.

\subsection{Width of sets and purely unrectifiable sets}

\begin{definizione}[Cones, proper cones]
An open set $C\subseteq \R^n$ is said to be a \emph{cone} if it is invariant under dilations and convex.

Let $\sigma\in(0,1)$ and $e\in\mathbb{S}^{n-1}$. The \emph{proper cone} of axis $e$ and amplitude $\sigma$ in $\R^n$ is the set:
\begin{equation}
C(e,\sigma):=\left\{x\in\R^n:\langle x,e\rangle>(1-\sigma)\lvert x\rvert\right\}.
\end{equation}
\end{definizione}

\begin{definizione}[Curves going in the direction of a cone]
\label{curv}
Let $\sigma\in(0,1)$ and $e\in\mathbb{S}^{n-1}$. Suppose $I\subseteq \R$ is a bounded interval and $\gamma:I\to \R^n$ is a Lipschitz map. We say that $\gamma$ is a  curve which \emph{goes in the direction of the cone} $C$, or that it is a $C$-\emph{curve}, if for any $s,t\in I$ such that $t<s$ we have:
$$\gamma(s)-\gamma(t)\in C.$$
\end{definizione}

\begin{figure}[ht!]
    \centering
    \includegraphics[scale=0.2]{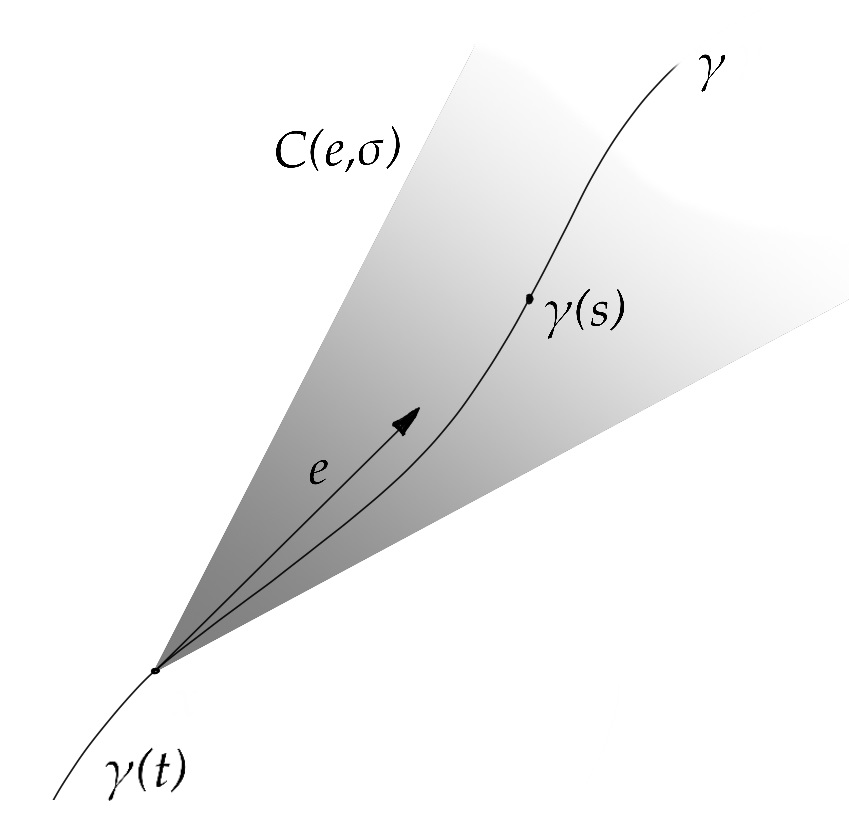}\qquad
    \caption{A $C(e,\sigma)$-curve $\gamma$.}
\end{figure}

\begin{definizione}
A Borel set $E$ is said to be \emph{purely unrectifiable} if for any Lipschitz curve $\gamma:\R\to\R^n$ we have:
$$\mathcal{H}^1(\gamma(\R)\cap E)=0.$$
\end{definizione}

The following proposition gives us an easy way to characterize pure unrectifiability (as curves with almost fixed derivative are easier to deal with then general Lipschitz curves).

\begin{proposizione}\label{equiv}
Let $E\subset\R^n$ be Borel. The following are equivalent:
\begin{itemize}
\item[(i)] $E$ is purely unrectifiable,
\item[(ii)]for any $\sigma>0$, $e\in\mathbb{S}^{n-1}$ and any $C(e,\sigma)$-curve $\gamma$ we have 
$\mathcal{H}^1(E\cap\gamma(I))=0$.
\end{itemize}
\end{proposizione}

Fixed a proper cone $C(e,\sigma)$ for any $C(e,\sigma)$-curve $\gamma$ we wish to fix a standard parametrisation for the image of $\gamma$.

\begin{definizione}
Let $C(e,\sigma)$ be a proper cone and $\gamma:I=(a,b)\to\R^n$ be a $C(e,\sigma)$-curve. Defined $T:=\langle e, \gamma(b)-\gamma(a)\rangle$, we say that $\tilde{\gamma}:\tilde{I}=(0,T)\to \R^n$ is the canonical parametrization of $\gamma(I)$ with respect to $C(e,\sigma)$ if $\tilde{\gamma}(\tilde{I})=\gamma(I)$ and:
$$\tilde{\gamma}(s)=s e+\eta(s),$$
where $\eta:\tilde{I}\to\R^n$ is a curve such that $\eta(0)=\gamma(a)$ and $\eta(s)-\eta(t)\in e^\perp$ for any $s,t\in (0,T)$.
\end{definizione}

\begin{osservazione}
In the following we will often omit the cone with respect to which the canonical representation is constructed, as it will be clear from the context. Note moreover that the canonical parametrization is unique.
\end{osservazione}

The following proposition shows that the canonical parametrisation always exists:

\begin{proposizione}\label{canone}
For any $C(e,\sigma)$-curve $\gamma:I\to\R^n$, the canonical parametrization $\tilde{\gamma}$ of $\gamma$ with respect to the cone $C(e,\sigma)$ exists and:
\begin{itemize}
\item[(i)]$\lvert\eta(t)-\eta(s)\rvert\leq \beta(\sigma)\lvert t-s\rvert$.
\item[(ii)]$\lvert\tilde{\gamma}(t)-\tilde{\gamma}(s)\rvert\leq\left(1+\beta(\sigma)\right)\lvert t-s\rvert$
\end{itemize}
where $\beta(\sigma):=\sqrt{2\sigma-\sigma^2}/(1-\sigma)$.
\end{proposizione}

The following proposition will become useful in the proof of Proposition \ref{princ}. It gives a uniform upper bound on the diameter of the parameter space of any canonical parametrization of any $C(e,\sigma)$-curve with values in the unit cube. We omit the proof, which is achieved by contradiction:

\begin{proposizione}\label{diam}
Let $\gamma:(0,T)\to[0,1]^n$ be the canonical parametrization of a Lipschitz curve going in the direction of $C(e,\sigma)$.
Then $T\leq \sqrt{n}$.
\end{proposizione}

The notion of directional width of sets was introduced by G. Alberti, M. Cs\"orney and D. Preiss in \cite{ACP0.2}. We give here a simplified version of the original definition:

\begin{definizione}[Directional width of a set]\label{spessore}
Let $E\subseteq [0,1]^n$ be a Borel set. We define the width $w_{C(e,\sigma)}[E]$ of the set $E$ along the proper cone $C(e,\sigma)$ as:
\begin{equation}
w_{C(e,\sigma)}[E]:=\inf_{\substack{E\subseteq G\\ G\text{ open}}} \sup_{\gamma\in\Gamma_{e,\sigma}}\int_{\substack{t\in I\\ \gamma(t)\in G}} \lvert\gamma^\prime(t)\rvert dt,
\nonumber
\end{equation}
where $\Gamma_{e,\sigma}$ is the set of the canonical parametrizations of those $C(e,\sigma)$-curves $\gamma:I\to[0,1]^n$ which are piece-wise affine and maximal, i.e., if $\tilde{\gamma}:J\to[0,1]^n$ is a $C(e,\sigma)$-curve extending $\gamma$, then $\tilde{\gamma}=\gamma$.

If $w_{C(e,\sigma)}[E]=0$, we say that $E$ is $C(e,\sigma)$-\emph{null}.
\end{definizione}

\begin{osservazione}
In Definition \ref{spessore} if the curves of $\Gamma_{e,\sigma}$ are not required maximal the definition of $\omega_{C(e,\sigma)}$ does not change. This assumption however will make the proof of Proposition \ref{pippol} easier.
\end{osservazione}

\begin{osservazione}
The width $\omega_C$ is monotone with respect to inclusion of sets, i.e., if $E\subseteq F$ then $\omega_{C(e,\sigma)}(E)\leq \omega_{C(e,\sigma)}(F)$.
Moreover, thanks to Proposition \ref{diam}, if $E\Subset (0,1)^n$ we have the bound:
\begin{equation}
    \begin{split}
    w_{C(e,\sigma)}[E]\leq\int_{\substack{t\in I\\ \gamma(t)\in (0,1)^n}} \lvert\gamma^\prime(t)\rvert dt\leq \sqrt{n}(1+\beta(\sigma)).
        \nonumber
    \end{split}
    \end{equation}
Furthermore if $G^\prime\subseteq(0,1)^n$ is an open set containing $E$, then:
\begin{equation}
w_{C(e,\sigma)}[E]=\inf_{\substack{E\subseteq G\subseteq G^\prime\\ G\text{ open}}}\sup_{\gamma\in\Gamma_{e,\sigma}}\int_{\substack{t\in I\\ \gamma(t)\in G}} \lvert\gamma^\prime(t)\rvert dt.
\nonumber
\end{equation}
\end{osservazione}

Proposition \ref{spessspess} will be of capital importance in proof of the implication (i)$\Rightarrow$(ii) of Theorem \ref{prinprin}. The reason for which is so useful is that it shows that $w_{C(e,\sigma)}$ enjoys some kind of $\sigma$-finiteness:

\begin{proposizione}
\label{spessspess}
Let $A,E\Subset (0,1)^n$ be Borel sets such that $w_{C(e,\sigma)}[E]<w_{C(e,\sigma)}[A]$. Then:
$$w_{C(e,\sigma)}[A\setminus E]>0.$$
\end{proposizione} 

\begin{proof}
Since $w_{C(e,\sigma)}[E]<w_{C(e,\sigma)}[A]$, there exists a $\delta>0$ for which the inequality persits:
\begin{equation}
w_{C(e,\sigma^\prime)}[E]+\delta<w_{C(e,\sigma)}[A].\nonumber
\end{equation}
Let $\epsilon<\delta/2$ and $\Omega\subseteq (0,1)^n$ be an open set such that $E\subseteq \Omega$ and $w_{C(e,\sigma)}[E]+\epsilon\geq w_{C(e,\sigma)}[\Omega]$. Then:
\begin{equation}
\begin{split}
w_{C(e,\sigma)}&[A\setminus E]\geq w_{C(e,\sigma)}[A\setminus \Omega]
=\inf_{\substack{A\setminus \Omega\subseteq G\\ G\text{ open}}}\sup_{\gamma\in\Gamma_{e,\sigma}}\int_{\substack{t\in I\\ \gamma(t)\in G}} \lvert\gamma^\prime(t)\rvert dt\\
\geq&\inf_{\substack{A\setminus \Omega\subseteq G\\ G\text{ open}}}\sup_{\gamma\in\Gamma_{e,\sigma}}\bigg(\int_{\substack{t\in I\\ \gamma(t)\in G\cup\Omega}} \lvert\gamma^\prime(t)\rvert dt-\int_{\substack{t\in I\\ \gamma(t)\in \Omega}} \lvert \gamma^\prime(t)\rvert dt\bigg)\\
\geq& \inf_{\substack{A\subseteq G\\ G\text{ open}}}\sup_{\gamma\in\Gamma_{e,\sigma}}\int_{\substack{t\in I\\ \gamma(t)\in G}} \lvert \gamma^\prime(t)\rvert dt-\sup_{\gamma\in\Gamma_{e,\sigma}}\int_{\substack{t\in I\\ \gamma(t)\in \Omega}} \lvert \gamma^\prime(t)\rvert dt
=w_{C(e,\sigma)}[A]-w_{C(e,\sigma)}[\Omega].\nonumber
\nonumber
\end{split}
\end{equation}
This implies that:
\begin{equation}
\begin{split}
w_{C(e,\sigma)}[((0,1)^n\setminus E)\cap A]\geq& w_{C(e,\sigma)}[A]-w_{C(e,\sigma)}[\Omega]
\geq w_{C(e,\sigma)}[A]-w_{C(e,\sigma)}[E]-\epsilon\geq\delta/2.\nonumber
\end{split}
\end{equation}
\end{proof}


The following proposition insures us that if a set is null with respect to a finite family of cones $\{C_i\}_{i=1,\ldots,N}$, then it is null with respect every other cone $C$ contained in $\bigcup_{i=1}^N C_i$. We omit the proof, which will appear in \cite{ACP0.2}.

\begin{proposizione}\label{salvato}
Suppose $E\subseteq [0,1]^n$ and that $w_{C(e,\sigma)}[E]>0$. Then for any $\sigma^\prime<\sigma$ there exists $e^\prime\in C(e,\sigma)$ such that $w_{C(e^\prime,\sigma^\prime)}[E]>0$.
\end{proposizione}


The definition of width of a set allows us to introduce the following notion of pure unrectifiability, which was first given in \cite{ACP0.2}:

\begin{definizione}[Uniform pure unrectifiability]
A Borel set $E\subseteq \R^n$ is said to be \emph{uniformily purely unrectifiable} if $w_{C(e,\sigma)}[E]=0$ for any $e\in\mathbb{S}^{n-1}$ and any $\sigma>0$.
\end{definizione}

In \cite{ACP0.2} it is shown that given a uniformly purely unrectifiable set $E$, one can construct a Lipschitz function being non-differentiable along any direction on $E$. One of the big challenges of that paper was to prove that a purely unrectifiable set is also uniformly purely unrectifiable. We give here a weaker version of this equivalence, which will suffice for our purposes:

\begin{proposizione}\label{equivequiv}
Let $E\subseteq \R^n$ be a compact set. The following are equivalent:
\begin{itemize}
\item[(i)] $E$ is uniformily purely unrectifiable,
\item[(ii)] $E$ is purely unrectifiable.
\end{itemize}
\end{proposizione}

\begin{proof}
$(i)\Rightarrow (ii)$ Thanks to Proposition \ref{equiv}, it is sufficient to prove that for any $e\in\mathbb{S}^{n-1}$, any $\sigma>0$ and any $C(e,\sigma)$-curve $\gamma:I\to \R^n$, we have that $\mathcal{H}^1(E\cap \gamma(I))=0$. Let $\tilde{\gamma}:(0,T)\to\R^n$ be the canonical parametrization of $\gamma(I)$ with respect to the cone $C(e,\sigma)$. Note that:
\begin{equation}
\begin{split}
    \mathcal{H}^1(\gamma(I)\cap G_\epsilon)=&\mathcal{H}^1(\tilde{\gamma}(\tilde{I})\cap G_\epsilon)=\int_{\substack{t\in\tilde{I}\\ \tilde{\gamma}{(t)}\in G_{\epsilon}}} \lvert \tilde{\gamma}^\prime(t)\rvert dt
    \leq (1+\beta(\sigma))\mathcal{L}^1(\{t\in \tilde{I}:\tilde{\gamma}\in G_\epsilon\}).
    \nonumber
\end{split}
\end{equation}
Consider now a sequence $\gamma_k:\tilde{I}\to \R^n$, which is given by linear interpolation of $\tilde{\gamma}$ and for which $\lVert \tilde{\gamma}-\gamma_k\rVert_{\infty,\tilde{I}}\leq 1/k$. By the uniform pure unrectifiability of $E$, for any $\epsilon>0$ there exists an open set $G_{\epsilon}\supseteq E$ such that:
\begin{equation}
\mathcal{L}^1(\{t\in\tilde{I}: \gamma_k(t)\in G_\epsilon\})\leq\int_{\substack{t\in\tilde{I}\\ \gamma_k(t)\in G_{\epsilon}}} \lvert \gamma_k^\prime(t)\rvert dt\leq \epsilon\nonumber,
\end{equation}
for any $k\in\N$.  Let $s\in \tilde{I}$ and suppose that $B_\delta(\tilde{\gamma}(s))\subseteq G_\epsilon$. Then for any $k>1/\delta$, $\gamma_k(s)\in G_\epsilon$. This implies that:
$$\chi_{\{t\in\tilde{I}: \tilde{\gamma}(t)\in G_\epsilon\}}(s)\leq \liminf_{k\to \infty} \chi_{\{t\in\tilde{I}: \gamma_k(t)\in G_\epsilon\}}(s).$$
Therefore by Fatou's lemma, we deduce that:
\begin{equation}
    \begin{split}
       \mathcal{L}^1(\{t\in\tilde{I}: \tilde{\gamma}(t)\in G_\epsilon\})\leq&\int\liminf_{k\to \infty} \chi_{\{t\in\tilde{I}: \gamma_k(t)\in G_\epsilon\}}(s)ds
       \leq\liminf_{k\to \infty}\mathcal{L}^1(\{t\in\tilde{I}: \gamma_k(t)\in G_\epsilon\}).
        \nonumber
    \end{split}
\end{equation}
Summing up, we have:
\begin{equation}
    \begin{split}
        \mathcal{H}^1(\gamma(I)\cap G_\epsilon)\leq& (1+\beta(\sigma))\liminf_{k\to \infty}\mathcal{L}^1(\{t\in\tilde{I}: \gamma_k(t)\in G_\epsilon\})
        \leq(1+\beta(\sigma))\epsilon.
        \nonumber
    \end{split}
\end{equation}
By arbitrariness of $\epsilon$ we conclude.

$(ii)\Rightarrow (i)$ In Proposition 7.4 of \cite{AlbertiMarchese} it has been proved that provided $E\subseteq\R^n$ is a compact set and for every $C(e,\sigma)$-curve $\gamma:I\to\R^n$ one has $\mathcal{H}^1(E\cap \gamma(I))=0$, then for any $\epsilon>0$ there exists a $\delta>0$ for which $\mathcal{H}^1(B_\delta(E)\cap \gamma(I))\leq\epsilon$. Since $E$ is compact and purely unrectifiable this property holds for every cone $C(e,\sigma)$. Thus for any $C(e,\sigma)$-curve this yields:
\begin{equation}
\int_{\substack{t\in I\\ \gamma(t)\in B_\delta(E)}} \lvert\gamma^\prime(t)\rvert dt=\mathcal{H}^1(\gamma(I)\cap B_\delta(E))\leq\epsilon.\nonumber
\end{equation}
In particular, we deduce that $w_{C(e,\sigma)}[E]=0$,
for any $e\in\mathbb{S}^{n-1}$ and $\sigma\in(0,1)$.
\end{proof}	

One of the key steps needed for the proof of the implication (ii)$\Rightarrow$(i) of Theorem \ref{prinprin}, is the understanding of the structure of Borel sets which cannot be covered by countably many purely unrectifiable sets. The following result characterise such sets:

\begin{teorema}\label{tSoleki}
Let $E\subseteq \R^n$ be an analytic set. Then:
\begin{itemize}
    \item[(i)]either $E$ is covered by a countable union of closed uniformily purley unrectifiable sets,
    \item[(ii)]or there exists a closed set $F$ such that:
    \begin{itemize}
        \item[($\alpha$)]$E\cap F$ contains a $G_\delta$ set dense in $F$,
        \item[($\beta$)] for every open set $U$ such that $U\cap F\neq \emptyset$ we have that $\cl(U\cap F)$ has positive width with respect to some cone.
    \end{itemize}
\end{itemize}
\end{teorema}

\begin{proof}
Just apply Theorem 2 of \cite{Coveringofclosedsets}, where in this case the $\sigma$-ideal $I_{ext}$ is the family of $F^\sigma$ purely unrectifiable subsets of $[0,1]^n$ and $I$ is the family of compact purely unrectifiable subsets of $[0,1]^n$.
\end{proof}

\begin{osservazione}
Note that in the above proposition, point (ii)($\beta$) can be strengthened to:
    \begin{itemize}
    \leftskip=1cm
        \item [($\beta^\prime$)]\emph{for every open set $U$ such that $U\cap F\neq \emptyset$ we have that $U\cap F$ has positive width with respect to some cone.}
\end{itemize}
This can be done just by observing that since $U\cap F\neq \emptyset$, then there must exist a small open ball $B$ compactly contained in $U$ such that $\cl(B\cap F)\subseteq U\cap F$. This and the fact that $\cl(B\cap F)$ must have positive width imply the claim. 
\end{osservazione}

\begin{definizione}\label{defiwi}
Assume $F$ is a Borel set, $\sigma>0$ and $u\in\mathbb{S}^{n-1}$. We define:
\begin{equation}
\mathcal{A}_F(u,\sigma):=\{x\in F:\omega_{C(u,\sigma)}(B_r(x)\cap F)>0\text{ for any }r>0 \}.
\label{Asets}
\end{equation}
The set $F$ is said to \emph{have every portion of positive }$C(u,\sigma)$-\emph{width} if $F\subseteq\mathcal{A}_F(u,\sigma)$.
\end{definizione}

The following proposition is an application of Proposition \ref{salvato}, and it will be used in the proof of Proposition \ref{BMstrategia}.

\begin{proposizione}\label{splitter}
Suppose $F$ is a closed set with every portion of positive $C(u,\sigma)$-width and let $\{u_i\}_{i=1,\ldots, N}$ be a $\sigma/16$-dense set in $\mathbb{S}^{n-1}$, i.e., for any $x\in\mathbb{S}^{n-1}$ there is $i\in\{1,\ldots,N\}$ such that $\lvert x-u_i\rvert<\sigma/16$. Then: 
$$F=\bigcup_{i=1}^N \mathcal{A}_F(u_i,\sigma/8).$$
For every $i=\{1,\ldots, N\}$ the set $\mathcal{A}_F(u_i,\sigma/8)$ is closed and if it is non-empty, it has every portion of positive $\omega_{C(u_i,\sigma/8)}$-width.
\end{proposizione}

\begin{proof}
Since $F$ has any portion of positive $\omega_{C(u,\sigma)}$-width, for any $x\in F$ and any $r>0$ we have $\omega_{C(u,\sigma)}(B_r(x)\cap F)>0$, and thus Proposition \ref{salvato} implies that there exists $u_i\in C(u,\sigma)$ such that $\omega_{C(u_i,\sigma/8)}(B_r(x)\cap F)>0$. This implies that $F$ is covered by the sets $\{\mathcal{A}_F(u_i,\sigma/8)\}_{i=1,\ldots, N}$.

Let $\{x_j\}_{j\in\N}$ be a sequence contained in $F$, converging to some $x\in \R^n$. Since $F$ is closed, $x\in F$ and for any $\rho>0$ we have $B_{\rho/3}(x_i)\subseteq B_\rho(x)$ for a sufficiently big $i$. This implies that:
$$0<\omega_{C(u_i,\sigma/8)}(F\cap B_{\rho/3}(x_i))\leq \omega_{C(u_i,\sigma/8)}(F\cap B_{\rho}(x)),$$
and therefore $x\in\mathcal{A}_F(u_i,\sigma/8)$.

We are left to prove that if $\mathcal{A}_F(u_i,\sigma/8)$ is non-empty, then $\mathcal{A}_F(u_i,\sigma/8)$ has every portion of positive $\omega_{C(e,\sigma)}$-width. Assume by contradiction there exists an open set $U$ such that $U\cap \mathcal{A}_F(u_i,\sigma/8)\neq \emptyset$ and:
\begin{equation}
    \omega_{C(u,\sigma)}(U\cap \mathcal{A}_F(u_i, \sigma/8))=0.
\label{equl}
\end{equation}
For any $\epsilon>0$ there exists an open neighbourhood $V_\epsilon$ of $\mathcal{A}_F(u_i,\sigma/8)$ such that $\omega_{C(e,\sigma)}(U\cap V_\epsilon)<\epsilon$.
Moreover for any $y\in F\setminus \mathcal{A}_F(u_i,\sigma/8)$ there exists $\rho(y)>0$ such that $\omega_{C(u_i,\sigma/8)} (B_{\rho(y)}(y))=0$. Since $\text{cl}(U)\cap F$ is compact we can find many $y_1,\ldots,y_N\in F\setminus \mathcal{A}_F(u_i,\sigma/8)$ such that:
$$\text{cl}(U)\cap F\subseteq V_\epsilon \cup \bigcup_{i=1}^N B_{\rho(y_i)}(y_i),ì.$$
Thanks to the subadditivity of the width, we deduce that:
$$\omega_{C(u,\sigma)}(U\cap F)\leq \omega_{C(u,\sigma)}(U\cap V_\epsilon)+\sum_{i=1}^N \omega_{C(u,\sigma)}(B_{\rho(y_i)}(y_i))\leq \epsilon.$$
By arbitrariness of $\epsilon$ we contradict the fact that $x\in\mathcal{A}_F(u_i,\sigma/8)$.
\end{proof}

\subsection{Piece-wise congruent mappings}\label{deffi}
\label{SS1.3}
In this subsection we always suppose that $m\geq n$.
Since our main result Theorem \ref{prinprin} requires the definition of a topology on Lipschitz functions, we introduce here the ambient space:
\begin{definizione}
We let $\lip([0,1]^n,\R^m)$ be the complete metric space of $1$-Lipschitz functions from $[0,1]^n$ to $\R^m$ endowed with the topology induced by uniform convergence.
\end{definizione}

We introduce now a subset of $\lip([0,1]^m,\R^m)$ with nice rigidity properties. In order to define these maps we need to introduce some notation. We define $\tau$ to be the collection of all finite families of open symplexes $\Pi\subseteq\mathcal{P}([0,1]^n)$ which are pairwise disjoint and:
$$\bigcup\{\cl(P):P\in\Pi\}=[0,1]^n.$$

\begin{definizione}[piece-wise congruent mappings]\label{pwcongr}
A mapping $f:[0,1]^n\to \R^m$ is said to be piece-wise congruent if
it is continuous and there exists a $\Pi\in\tau$ such that the restriction $f\rvert_P$ for
any $P\in\Pi$ is an affine isometric mapping, i.e. there are $A_P\in O(n,m)$ and $b_P\in\R^m$ such that:
 $$f(x)=A_Px+b_P,$$
for any $x\in P$. The set of all piece-wise congruent mappings will be denoted by $\mathfrak{P}(n,m)$, and a partition $\Pi\in\tau$ for which $f\lvert_P$ is affine for any $P\in\Pi$ is said \emph{adapted to} $f$.
\end{definizione}

In order to prove that the set $\mathfrak{P}(n,m)$ is dense in $\lip([0,1]^n,\R^m)$, we recall that U. Brehm in \cite{PLisometries} proved the following:

\begin{teorema}
Let $M\subseteq \R^n$ be a finite set and let $f:\R^n\supseteq M\to\R^m$ be a $1$-Lipschitz function on $M$, i.e.:
$$\lvert f(x)-f(y)\rvert\leq\lvert x-y\rvert,$$ 
for any $x,y\in M$. Then $f$ has a piece-wise congruent extension $\overline{f}:[0,1]^n\to \R^m$ . \label{PLiso}
\end{teorema}

As a corollary we deduce that:

\begin{corollario}\label{dense}
$\mathfrak{P}(n,m)$ is dense in $\lip([0,1]^n,\R^m)$.
\end{corollario}

\begin{proof}
Fix any $f\in\lip([0,1]^n,\R^m)$ and define $M:=(1/2^k)\Z^n\cap[0,1]^n$. Then Theorem \ref{PLiso} implies that we can find a piecewise conguent mapping $g:[0,1]^n\to\R^m$ which coincides with $f$ on $M$. However since $M$ is $\sqrt{n}/2^k$-dense in $[0,1]^n$, we deduce that $\lVert f-g\rVert_\infty\leq \sqrt{n}/2^{k-1}$.
\end{proof}

\begin{osservazione}
The space of piece-wise congruent mappings will play a fundamental role in the proof of the implication (ii)$\Rightarrow$(i) of Theorem \ref{prinprin}, as already explained in the introduction. From the technical point of view they are so important since the composition of a function in $\mathfrak{P}(n,m)$ with piece-wise affine curves is still a piecewise congruent mapping from the line to $\R^m$. This will allow us in Section \ref{Ss4} to prove (with an approach close in the ideas to the proof of the Theorem \ref{isp} in \cite{PreissTiser} by Preiss and Ti\u{s}er) that if two piece-wise congruent mappings are close in the supremum norm, then the set where they have directional derivative along a fixed direction $e$ which are not close has small width with respect to a small cone of axis $e$.
\end{osservazione}

\begin{definizione}[$\epsilon$-differentiability]
A Lipschitz function $f:\R^n\to\R^m$ is said to be non-$\epsilon$-differentiable at $x$ along $e$ if for any $d\in\R^m$ we have:
$$\limsup_{t\to 0} \left\lvert\frac{f(x+te)-f(x)-dt}{t}\right\rvert> \epsilon.$$
On the other hand $d\in\R^m$ is said to be an $\epsilon$-derivative of $f$ at $x\in\R^n$ along $e\in\mathbb{S}^{n-1}$ if:
$$\limsup_{t\to 0} \left\lvert\frac{f(x+te)-f(x)-dt}{t}\right\rvert\leq \epsilon.$$
\end{definizione}

The following proposition shows that the set where a piece-wise congruent mapping is non-$\epsilon$-differentiable along a direction $e$ has small width with respect to a cone with small enough (with respect to $\epsilon$) aperture and axis $e$.

\begin{proposizione}
Fix a piece-wise congruent mapping $f\in\mathfrak{P}(n,m)$, a direction $e\in\mathbb{S}^{n-1}$ and let $\sigma\in(0,1/10)$. For any $\sqrt{2\sigma}<\epsilon<1/10$, the set:
\begin{equation}
\Xi(f,e,\epsilon):=\{x\in(0,1)^n:f \text{ is not }\epsilon\text{-directionally differentiable along }e\},\label{F1}
\end{equation}
has $C(e,\sigma)$-null closure, i.e. $w_{C(e,\sigma)}(\text{cl}(\Xi(f,e,\epsilon)))=0$.
\label{fortuna}
\end{proposizione}

\begin{proof} Let $\Pi$ be a partition adapted to $f$. For any $k\in\{0,\ldots,n-1\}$ let $\mathcal{F}_k$ be the family of all the $k$-dimensional faces of every $P\in\Pi$, and we let:
$$\mathcal{S}_k:=\{\Delta\in\mathcal{F}_k:\Delta \text{ is a }C(e,\sigma)\text{-null set}\}.$$
If we prove that:
\begin{equation}
    \Xi(f,e,\epsilon)\subseteq \bigcup_{k=0}^{n-1}\bigcup \mathcal{S}_k=:\mathcal{N},
    \label{cl2}
\end{equation}
the thesis of the proposition would follow immediately as $\mathcal{N}$ is the finite union of closed $C(e,\sigma)$-null sets. Since $f$ is not differentiable at every point of $\Xi(f,e,\epsilon)$, we deduce that $\Xi(f,e,\epsilon)\subseteq \bigcup_{k=0}^{n-1}\bigcup \mathcal{F}_k=:\mathcal{M}$. Hence for any $x\in\mathcal{M}\setminus\mathcal{N}$, we define:
$$d(x):=\min\{k:\text{ there exists  a }\Delta\in\mathcal{F}_k\setminus \mathcal{S}_k\text{ such that }x\in\Delta\}.$$
By the minimality of $d(x)$ for any $\Delta\in\mathcal{F}_k\setminus\mathcal{S}_k$ such that $x\in\Delta$, the point $x$ does not belong to the $(d(x)-1)$-skeleton of $\Delta$ and thus there exist $v\in C(e,\sigma)$ and $\delta>0$ such that $x+v[-\delta,\delta]\subseteq \Delta$. If this was not the case the $k$-plane spanned by $\Delta$ would be $C(e,\sigma)$-null and this would contradict the fact that $\Delta\in\mathcal{F}_k\setminus\mathcal{S}_k$. Since $f$ is piece-wise congruent and $x+v[-\delta,\delta]\subseteq \text{cl}(P)$ we deduce that the restriction of $f$ to the segment $x+v[-\delta,\delta]$ is linear and thus differentiable along $v$ at $x$. Therefore:
\begin{equation}
\begin{split}
\limsup_{h\to 0}\left\lvert\frac{f(x+he)-f(x)-h\partial_v f(x)}{h}\right\rvert
\leq\left\lvert v-e\right\rvert\leq(2\sigma)^\frac{1}{2}<\epsilon,
\nonumber
\end{split}
\end{equation}
which implies that $\partial_v f(x)$ is an $\epsilon$-derivative of $f$ at $x$. Since $x\in\mathcal{M}\setminus \mathcal{N}$ was arbitrary, the inclusion \eqref{cl2} is proved.
\end{proof}

\section{The main result}
\label{S2}

Let $f:\R^n\to\R^m$ be a Lipschitz mapping and $x\in\R^n$. We say that $f$ is \emph{fully non-differentiable} at $x$ if for any $e\in\mathbb{S}^{n-1}$ the limit:
\begin{equation}
\lim_{t\to 0}\frac{f(x+te)-f(x)}{t},\nonumber
\end{equation}
does not exists.
We define $f$ to be fully \emph{non-differentiable on a set $E\subseteq \R^n$} if and only if it is fully non-differentiable at every $x\in E$. Moreover we define the \emph{derived set} of $f$ at $x$ as:
\begin{equation}
\mathcal{D}f(x,e):=\bigcap_{r>0}\text{cl}\left({\left\{\frac{f(x+te)-f(x)}{t}:0<t<r\right\}}\right).\nonumber
\end{equation}

The precise statement of our main result Theorem \ref{prinprin} is the following:

\begin{teorema}
\label{main}
Let $E\Subset [0,1]^n$ be an analytic set and let $n\leq m$. The following are equivalent:
\begin{itemize}
\item[(i)]$E$ is contained in a countable union of closed purely unrectifiable sets,
\item[(ii)] the set $S$ of fully non-differentiable maps on $E$ is residual in $\lip([0,1]^n,\R^m)$.
\item[(iii)] the set:
\begin{equation}
S^\prime:=\left\{f\in\lip([0,1]^n,\R^m):\mathcal{D}f(x,v)\supseteq\mathbb{S}^{m-1}\text{ for any }v\in\mathbb{S}^{n-1}\right\}\nonumber,
\end{equation}
is residual in $\lip([0,1]^n,\R^m)$.
\end{itemize}
\end{teorema}

We state here two Propositions which toghether imply Theorem \ref{main}. Their proofs are postponed to Section \ref{albalb} and Section \ref{Ss4}.

\begin{restatable}{proposizione}{evev}
Let $k_0\in\N$, $E\subseteq\left[1/k_0,1-1/k_0\right]^n$ be a compact purely unrectifialbe set,  $v\in\mathbb{S}^{n-1}$ and $w\in\mathbb{S}^{m-1}$.
For any $k>k_0$ we define $S_{k}^{v,w}$ to be subset of those $f\in\lip([0,1]^n,\R^m)$ for which there exists a $\delta=\delta(f)<1/k$ such that for any $x\in E$ there is a $\tau=\tau(x)\in(-1/k,-\delta)\cup(\delta,1/k)$ which satisfy:
$$\left\lvert\frac{f(x+\tau e)-f(x)-\tau w}{\tau}\right\rvert<\frac{1}{k}-\delta.$$
Then $S_{k}^{v,w}$ is open and dense in $\lip([0,1]^n,\R^m)$.
\label{evev}
\end{restatable}

\begin{restatable}{proposizione}{ueue}
\label{BMstrategia}
Let $k_0\in\N$, $n\leq m$ and suppose $F\subseteq [1/k_0,1-1/k_0]^n$ is a closed non-empty set such that:
\begin{itemize}
\item[(i)] there exists some Borel set $E\subseteq[1/k_0,1-1/k_0]^n$ for which $E\cap F$ is residual in $F$,
\item[(ii)] for every open set $U$ such that $U\cap F\neq \emptyset$ we have that $U\cap F$ has positive width with respect to some cone.
\end{itemize}
Denote by $S\subseteq \lip([0,1]^n,\R^m)$ the set of functions which are differentiable along some direction at least a point of $E\cap F$. Then $S$ is residual in $\lip([0,1]^n,\R^m)$.
\end{restatable}

Despite the fact that the proofs of the above propositions have been postponed, we prove here that they actually imply Theorem \ref{main}.

\begin{proposizione}
Propositions \ref{evev} and \ref{BMstrategia} imply Theorem \ref{main}.
\end{proposizione}

\begin{proof}
First of all we tackle the implication (i)$\Rightarrow$(iii). Assume that $E$ is compact and let $k_0\in\N$ be such that $E\subset\left[1/k_0,1-1/k_0\right]^n$. Let $D^n$ and $D^m$ be countable dense subsets of $\mathbb{S}^{n-1}$ and $\mathbb{S}^{m-1}$ respectively. Define:
\begin{equation}
\tilde{S}:=\bigcap_{v\in D^n}\bigcap_{w\in D^m}\bigcap_{k>k_0} S_{k}^{v,w}\nonumber,
\end{equation}
where the sets $S_{k}^{v,w}$ were introduced in the statement of Proposition \ref{evev}.
Since countable intersection of residual sets is residual, if Proposition \ref{evev} holds true, then $\tilde{S}$ is residual and thus it suffices to prove that $\tilde{S}\subseteq S^\prime$ to conclude. For any $v\in\mathbb{S}^{n-1},w\in\mathbb{S}^{m-1}$ and any integer $i>k_0$ there are $v(i)\in D^n$ and $w(i)\in D^m$ such that $\lvert v-v(i)\rvert+\lvert w-w(i)\rvert\leq1/2^i$.
For any $f\in\tilde{S}$ and any $i\in\N$ there exists $t_i\in\left(-1/i,1/i\right)$ such that:
\begin{equation}
\begin{split}
\left\lvert\frac{f(x+t_iv(i))-f(x)-t_iw(i)}{t_i}\right\rvert<\frac{1}{2^i}\nonumber.
\end{split}
\end{equation}
Therefore we deduce that:
\begin{equation}
\begin{split}
   \left\lvert\frac{f(x+t_iv)-f(x)-t_iw}{t_i}\right\rvert<&\left\lvert \frac{f(x+t_iv(i))-f(x)-t_iw(i)}{t_i}\right\rvert+\lvert v-v(i)\rvert+\lvert w-w(i)\rvert\leq\frac{1}{2^{i-1}}.
\nonumber 
\end{split}
\end{equation}
Thanks to the above inequality, sending $i\to\infty$ we have:
\begin{equation}
\lim_{i\to\infty}\frac{f(x+t_iv)-f(x)}{t_i}=w\nonumber.
\end{equation}
This implies that for any $v\in\mathbb{S}^{n-1}$ and $w\in\mathbb{S}^{m-1}$ we have that $w\in \mathcal{D}f(x,v)$ and thus $\tilde{S}\subseteq S^\prime$.

Secondly we prove the implication (ii)$\Rightarrow$(i). Once again we can assume without loss of generality that $E\Subset [0,1]^n$ is compact, and thus there exists $k_0\in\N$ such that  $E\subseteq \left[1/k_0,1-1/k_0\right]^n$. Suppose by contraction that $E$ cannot be covered by countably many closed purely unrectifiable sets. By Theorem~\ref{tSoleki} we find a closed nonempty set $F\subseteq [0,1]^n$ with every portion of positive width such that $E\cap F$ is residual in $F$. Applying Proposition~\ref{BMstrategia} we get a contradiction.

Eventually the implication (iii)$\Rightarrow$(ii) is trivial since $S^\prime\subseteq S$.
\end{proof}

\begin{osservazione}
Since Proposition \ref{evev} holds even in the case $m\leq n$, the implication (i)$\Rightarrow$(iii) is true for any couple $n,m$.
\end{osservazione}

\section{Roughing of smooth functions on purely unrectifiable sets}
\label{albalb}

We introduce here the width function associated to an open set $\Omega\subseteq\R^n$. Such a function was first introduced in a more complex and general form in \cite{ACP0.2} and in the present, simplified form in \cite{AlbertiMarchese}. 

\begin{definizione}
Let $e\in\mathbb{S}^{n-1}$ and $\sigma>0$. We define the width function of $\Omega$ with respect to the cone $C(e,\sigma)$ as:
\begin{equation}
w_{C(e,\sigma)}[\Omega](x):=\sup_{\gamma\in\Gamma_{e,\sigma}^{x}}\{\mathcal{H}^1(\Omega\cap \gamma(I))-\lvert x-x_\gamma\rvert\},\nonumber
\end{equation}
where $\Gamma_{e,\sigma}^{x}$ is the set of canonical parametrizations of piece-wise affine $C(e,\sigma)$-curves $\gamma:(a,b)\to\R^n$ for some $a<b$ and  which endpoint $x_\gamma:=\gamma(b)$ is of the form $x_\gamma=x+se$ for some $s\geq 0$.
\end{definizione}

We recall here Proposition 7.4 of \cite{AlbertiMarchese}, which is a list of properties of the width function associated to a neighbourhood of a purely unrectifiable set:

\begin{proposizione}\label{alb}
Let $E$ be a purely unrectifiable compact subset of $(0,1)^n$. Fix any $e\in\mathbb{S}^{n-1}$ and $\sigma\in(0,1)$. For any $\epsilon>0$ there exists a $\delta>0$ such that:
\begin{equation}
\mathcal{H}^1(\gamma(I)\cap B_\delta(E))<\epsilon,\label{eq:2:2}
\end{equation}
for any $C(e,\sigma)$-curve $\gamma:I\to\R^n$.
Moreover the width function along $C(e,\sigma)$ associated to $B_\delta(E)$ has the following properties:
\begin{itemize}
\item[(i)] $0\leq w_{C(e,\sigma)}[B_\delta(E)](x)\leq \epsilon$ for any $x\in\R^n$,
\item[(ii)] $w_{C(e,\sigma)}[B_\delta(E)](x)\leq w_{C(e,\sigma)}[B_\delta(E)](x+se)\leq w_{C(e,\sigma)}[B_\delta(E)](x)+s$ for every $s>0$ and any $x\in\R^n$. Moreover, if the segment $[x,x+se]$ is contained in $B_\delta(E)$, then:$$w_{C(e,\sigma)}[B_\delta(E)](x+se)=w_{C(e,\sigma)}[B_\delta(E)](x)+s,$$
\item[(iii)] $\lvert w_{C(e,\sigma)}[B_\delta(E)](x+v)-w_{C(e,\sigma)}[B_\delta(E)](x)\rvert\leq\beta(\sigma)\lvert v\rvert$ for every $v\in V:=e^{\perp}$, where $\beta(\sigma)$ is the quantity defined in Proposition \ref{canone},
\item[(iv)]$w_{C(e,\sigma)}[B_\delta(E)](\cdot)$ is $1+\beta(\sigma)$-Lipschitz.
\end{itemize}
\end{proposizione}

In the following  $E\Subset (0,1)^n$ will always be a fixed compact purely unrectifiable set, $\epsilon$ and $\sigma$ two fixed positive numbers in $(0,1)$ and $\mathcal{E}:=\{e_1,\ldots,e_n\}$ an orthonormal basis of $\R^n$.

\begin{definizione}\label{defi}
Suppose $\delta_i$ with $i=1,\ldots,n$ are the numbers given by Proposition \ref{alb} for which \eqref{eq:2:2} holds for any $C(e_i,\sigma)$-curve $\gamma$. we define:
\begin{itemize}
\item[(a)]$\delta_{\mathcal{E}}:=\min\{\delta_1,\ldots,\delta_n,\epsilon\}/16$,
\item[(b)]$g_i(x):=w_{C(e_i,\sigma)}(B_{\delta_{\mathcal{E}}}(E))(x)/(1+\beta(\sigma))$,
\item[(c)] $G(x)=G^{\mathcal{E},\epsilon,\sigma}(x):=(g_1(x),\ldots,g_n(x)).$
\end{itemize}
\end{definizione}

In the following proposition we show that the gradient of the map $G$ introduced above on the set $E$ is very close to the identity matrix.

\begin{proposizione}
Let $G:\R^n\to\R^n$ be the function defined in point $(c)$ of Definition \ref{defi}. Then:
\begin{itemize}
\item[(i)]$\lVert G(x)\rVert_\infty\leq\sqrt{n}\epsilon/(1+\beta(\sigma))$
for every $x\in\R^n$.
\item[(ii)]For every $x\in B_{\delta_{\mathcal{E}}}(E)$ there exist $t^*=t^*(x)>0$ such that:
\begin{equation}
\left\lvert \frac{G(x+tv)-G(x)}{t}-v\right\rvert\leq 2(n+1)^\frac{1}{2}\frac{\beta(\sigma)}{1+\beta(\sigma)}\nonumber
\end{equation}
for all $t\in(-t^*,t^*)\setminus \{0\}$ and every $v\in\mathbb{S}^{n-1}$. Moreover if $x\in E$ then $t^*=\delta_{\mathcal{E}}/2$.
\end{itemize}
\label{ide}
\end{proposizione}

\begin{proof}
The estimate on the supremum norm of $G$ follows directly from Proposition \ref{alb}(i). In order to prove point (ii), we note that for any $x\in B_{\delta_{\mathcal{E}}}(E)$ we can find $t^*=t^*(x)>0$ such that $B_{t^*}(x)\subseteq B_{\delta_{\mathcal{E}}}(E)$. Thus Propostion \ref{alb}(ii) and (iii) imply that:
\begin{equation}
\begin{split}
&\lvert w_{C(e_i,\sigma)}(B_{\delta_{\mathcal{E}}}(E))(x+tv)-w_{C(e_i,\sigma)}(B_{\delta_{\mathcal{E}}}(E))(x)-tv_i\rvert\leq \beta(\sigma)\lvert t\rvert,
\nonumber
\end{split}
\end{equation}
for any $t\in(-t^*,t^*)$. Therefore, thanks to some algebraic computations, which we omit, we get:
\begin{equation}
\begin{split}
\left\lvert \frac{G(x+tv)-G(x)}{t}-v\right\rvert^2=&\frac{1}{t^2}\sum_{i=1}^n\left(\frac{w_{C(e_i,\sigma)}(B_{\delta_{\mathcal{E}}}(E))(x+tv)-w_{C(e_i,\sigma)}(B_{\delta_{\mathcal{E}}}(E))(x)}{1+\beta(\sigma)}-tv_i\right)^2\\
\leq&\frac{2}{t^2}\sum_{i=1}^n\left(\frac{\beta(\sigma)\lvert t\rvert}{1+\beta(\sigma)}\right)^2+\left(\frac{tv_i\beta(\sigma)}{1+\beta(\sigma)}\right)^2=2(n+1)\left(\frac{\beta(\sigma)}{1+\beta(\sigma)}\right)^2.
\end{split}\nonumber
\end{equation}
\end{proof}

Fix an open set $\Omega$ and convolve $\chi_{\Omega}$ with a smooth kernel. It is possible to construct for any $\varsigma>0$ a smooth function $\psi^\varsigma_{\Omega}:\R^n\to\R$ such that $\psi^\varsigma_{\Omega}(x)=1$ for every $x\in\Omega$, $\psi^\varsigma_{\Omega}(x)=0$ for every $x\in B_{\varsigma/2}(\Omega)^c$ and: 
\begin{equation}
    \text{Lip}(\psi^\varsigma_{\Omega})\leq 256\alpha(n)/\varsigma,
    \label{smoo}
\end{equation}
where $\alpha(n)$ is the volume of the unit ball in $\R^n$.
For any $v\in\mathbb{S}^{n-1}$ and $\lambda\in(0,1)$, we define:
\begin{equation}
H_v(x):=\frac{\psi_{B_\theta(E)}^{\delta_{\mathcal{E}}}(x)}{1+\beta(\sigma)}w_{C(v,\sigma)}\left[B_\theta(E)\right](x)\label{numbero1}
,
\end{equation}
where $\psi_{B_\theta(E)}^{\delta_{\mathcal{E}}}$ is the cutoff function described above and the parameter $\theta\in(0,\lambda\delta_{\mathcal{E}}/2)$ is chosen in such a way that:
\begin{equation}
w_{C(v,\sigma)}\left[B_\theta(E)\right](x)\leq \lambda\delta_{\mathcal{E}}/2.
\label{localoca}
\end{equation}

\begin{definizione}
Assume that $\epsilon,\sigma,\lambda,\theta$ and $v$ are as above and $f\in\mathcal{C}^\infty(\R^n,\R^m)$. For any $\eta\in(0,1/20)$ and $u\in\mathbb{S}^{m-1}$ we define the roughed function $F_u^\eta$ of $f$ as:
\begin{equation}
F_u^\eta(x):=(1-\eta)\left(f(x)-D f(x)[G(x)]+\frac{D^2 f(x)[G(x),G(x)]}{2}+H_{v}(x)u\right),\nonumber
\end{equation}
where $G$ is the map introduced in Definition \ref{defi}(c) and $H_v$ the function defined in \ref{numbero1}.

Note that the function $F^\eta_u$ depends on the parameters $\eta,\epsilon,\sigma,\lambda,\theta$ and on the directions $v\in\mathbb{S}^{n-1}$ and $u\in\mathbb{S}^{m-1}$, but for the sake of readability we have chosen this simplified notation.
\label{rough}
\end{definizione}

 We state here without proof the following lemma, which is a consequence of Lebesgue's differentiation theorem and which will be used in the proof of Proposition \ref{lung}. It states that the boundary of any neighbourhood of a compact set is $\sigma$-porous, and it will be used in the proof of Proposition \ref{lung}.

\begin{lemma}\label{surf}
Let $K\subseteq \R^n$ be a compact set and $\delta\in(0,1)$. For any $x\in\partial B_\delta(K)$ and any $\tau\in (0,1)$ there exist $x(\tau)\in(\partial B_\delta(K))^c$ such that $\lvert x(\tau)-x\rvert=\tau\delta$ and:
$$B_{\tau\delta/4}(x(\tau))\subset (\partial B_\tau(K))^c\cap B_{5\tau\delta/4}(x).$$
In particular $\mathcal{L}^n(\partial B_\delta(K))=0$.
\end{lemma}

\begin{proposizione}\label{lung}
Let $f\in\mathcal{C}^\infty(\R^n,\R^m)$ be a $1$-Lipschitz map. Then:
$$\lVert f-F_u^\eta\rVert_\infty\leq \eta\lVert f\rVert_\infty+(1-\eta) (n^\frac{1}{2}\epsilon+\lVert D^2 f\rVert_\infty n\epsilon^2+\lambda\epsilon).$$
Moreover, if we assume that: 
\begin{equation}
\sigma+\epsilon+\lambda<\min\bigg\{\frac{\eta^2}{2^{10}(n+1)^4},\frac{\eta}{2^{20}\alpha(n)},\frac{\eta}{2^{10}n\lVert D^2 f\rVert_\infty},\frac{\eta^\frac{1}{2}}{2^{10}n(1+\lVert D^3 f\rVert_\infty)}\bigg\},
\label{numbero2}
\end{equation}
we have:
\begin{itemize}
    \item[(i)]$\lVert D f_u^\eta(x)\rVert<(1-\eta^2)$, 
for  $\mathcal{L}^n$-almost every $x\in[0,1]^n$. In particular $F_u^\eta$ is a $1$-Lipschitz function.
\item[(ii)]There exists $0<h^*<\eta$ depending only on $\eta$, $f$ and $E$, such that:
\begin{equation}
\left\lvert\frac{F_u^\eta(x+hv)-F_u^\eta(x)}{h}-u \right\rvert\leq 10\eta-h^*,\label{numbero2001}
\end{equation}
for any $h\in (-2h^*,-h^*)\cup(h^*,2h^*)$ and any $x\in E$.
\end{itemize}
\end{proposizione}

\begin{proof} Applying the triangular inequality to the definition of $F_u^\eta$ we get:
\begin{equation}
\begin{split}
\lVert f-F_u^\eta\rVert_\infty\leq&\eta\lVert f\rVert_\infty
+(1-\eta)\Big(\lVert D f(\cdot)[G(\cdot)]\lVert_\infty+\lVert D^2 f(\cdot)[G(\cdot),G(\cdot)]\rVert_\infty+\rVert H_{v}(\cdot)u\rVert_\infty\Big).\label{numbero2002}
\end{split}
\end{equation}
Using the fact that $f$ is $1$-Lipschitz, the Cauchy-Schwartz inequality and the definition of $H_v$, we deduce that:
\begin{equation}
\begin{split}
\lVert D f(\cdot)[G(\cdot)]\rVert_\infty+&\lVert D^2 f(\cdot)[G(\cdot),G(\cdot)]\rVert_\infty+\lVert H_{v}(\cdot)u\rVert_\infty\\
\leq&\lVert G\rVert_\infty+\lVert D^2 f\rVert_\infty\lVert G\rVert_\infty^2+\frac{1}{1+\beta(\sigma)}\lVert \psi_{B_\theta(E)}^{\delta_{\mathcal{E}}}w_{C(e,\sigma)}\left[B_\theta(E)\right]\rVert_\infty\\
\leq&\lVert G\rVert_\infty+\lVert D^2 f\rVert_\infty\lVert G\rVert_\infty^2+\lVert w_{C(e,\sigma)}\left[B_\theta(E)\right]\rVert_\infty.
\nonumber
\end{split}
\end{equation}
Therefore, by \eqref{numbero2002}, Proposition \ref{ide} and the definition of $\theta$, see \eqref{localoca}, we have that:
$$\lVert f-F_u^\eta\rVert_\infty\leq \eta\lVert f\rVert_\infty+(1-\eta) (n^\frac{1}{2}\epsilon+\lVert D^2 f\rVert_\infty n\epsilon^2+\lambda\epsilon).$$

Assuming that \eqref{numbero2} holds, we prove now that the roughed function is $1$-Lipschitz. Let $B\subseteq(0,1)^n$ be the Borel set of full Lebesgue measure on which both $G$ and $H_{v}$ are differentiable and fix $x\in B$. The differential of $F_u^\eta$ at $x$ is given by the expression:
\begin{equation}
\begin{split}
\frac{D f_u^\eta(x)[\cdot]}{1-\eta}=
D f(x)[\cdot]-&D^2 f(x)[G(x),\cdot]-D f(x)D G(x)[\cdot]
+\frac{D^3 f(x)[G(x),G(x),\cdot]}{2}\\+&D^2 f(x)[D G(x)[\cdot],G(x)]+u\otimes \nabla H_{v}(x).
\end{split}
\nonumber
\end{equation}
Therefore the operatorial norm of $D f_u^\eta$ at $x$ can be bounded, using triangular inequality, by:
\begin{equation}
\begin{split}
\left\lVert \frac{D f_u^\eta(x)}{1-\eta}\right\rVert\leq\lVert D f(x)(\text{id}_{\R^n}-&D G(x))\rVert+\lVert D^2 f(x) \rVert\lVert G(x)\rVert
+\frac{\lVert D^3 f(x)\rVert \lVert G(x)\rVert^2}{2}\\+&\lVert D^2 f(x)\rVert \lVert D G(x)\rVert \lVert G(x)\rVert
+\lVert u \otimes \nabla H_{v}(x)\rVert.
\nonumber
\end{split}
\end{equation}
Proposition \ref{ide}(i), (ii) and the fact that $f$ and $G$ are $1$-Lipschitz imply that for every $x\in B$ we have:
\begin{equation}
\begin{split}
\left\lVert \frac{D f_u^\eta(x)}{1-\eta}\right\rVert\leq&\lVert \text{id}_{\R^n}-D G(x)\rVert+2\lVert D^2 f\rVert_\infty n^\frac{1}{2}\epsilon+\lVert D^3 f\rVert_\infty n\epsilon^2+\lVert u \otimes \nabla H_{v}(x)\rVert.
\label{numbero3}
\end{split}
\end{equation}
The only terms on the right-hand side of \eqref{numbero3} which need to be estimated are $\lVert \text{id}_{\R^n}-DG(x)\rVert$ and $\lVert u \otimes \nabla H_{v}(x)\rVert$.
Proposition \ref{surf} implies that $\mathcal{L}^n(\partial B_{2\delta_{\mathcal{E}}}(E))=0$, therefore without loss of generality we can assume either $x\in B\cap B_{2\delta_{\mathcal{E}}}(E)$ or $x\in B\cap (B_{\delta_{\mathcal{E}}}(E))^c\setminus \partial B_{\delta_{\mathcal{E}}}(E)$. We study $D f_u^\eta(x)$ in each of this cases.

Suppose first that $x\in B\cap B_{\delta_{\mathcal{E}}}(E)$. By Proposition \ref{ide} and the differentiability of $G$ at $x$ we have:
\begin{equation}
\begin{split}
&\lVert \text{id}_{\R^n}-D G(x)\rVert=\sup_{w\in\mathbb{S}^{n-1}}\lvert  w-D G(x)w\rvert
=\sup_{w\in\mathbb{S}^{n-1}}\lim_{t\to 0}\left\lvert  w-\frac{G(x+tw)-G(x)}{t}\right\rvert\leq\frac{(n+1)^\frac{1}{2}\beta(\sigma)}{1+\beta(\sigma)}.
\nonumber
\end{split}
\end{equation}
Thanks to the definition of $H_v$, we have that:
\begin{equation}
    u\otimes\nabla H_v=\frac{u\otimes\nabla\psi_{B_\theta(E)}^{\delta_{\mathcal{E}}}}{1+\beta(\sigma)}w_{C(v,\sigma)}\left[B_\theta(E)\right]+\frac{\psi_{B_\theta(E)}^{\delta_{\mathcal{E}}}}{1+\beta(\sigma)}u\otimes \nabla w_{C(v,\sigma)}\left[B_\theta(E)\right].
\label{labelotto}
\end{equation}
The estimate on $\lVert u\otimes \nabla H_v\rVert_\infty$ follows by \eqref{smoo}, \eqref{labelotto} and Proposition \ref{alb}, indeed:
\begin{equation}
\left\lVert\frac{ u\otimes \nabla\psi_{B_\theta(E)}^{\delta_{\mathcal{E}}}(x)}{1+\beta(\sigma)}w_{C(v,\sigma)}\left[B_\theta(E)\right](x)\right\rVert\leq\frac{\lambda\delta_{\mathcal{E}}}{1+\beta(\sigma)}\frac{256\alpha(n)}{\delta_{\mathcal{E}}}=\frac{256\alpha(n)\lambda}{1+\beta(\sigma)},
\nonumber
\end{equation}
and:
\begin{equation}
\left\lVert\frac{\psi_{B_\theta(E)}^{\delta_{\mathcal{E}}}(x)}{1+\beta(\sigma)}u\otimes \nabla w_{C(v,\sigma)}\left[B_\theta(E)\right](x)\right\rVert\leq 1\nonumber.
\end{equation}
Summing up, if $x\in B\cap B_{2\delta_{\mathcal{E}}}(E)$:
\begin{equation}
\begin{split}
    \lVert D f_u^\eta(x)\rVert\leq &(1-\eta)(2\lVert D^2 f\rVert_\infty n^\frac{1}{2}\epsilon
+\lVert D^3 f\rVert_\infty n\epsilon^2
+(n+1)^\frac{1}{2}\beta(\sigma)+256\alpha(n)\lambda+1)<(1-\eta^2),
\nonumber
\end{split}
\end{equation}
where the last inequality comes from the condition on $\epsilon+\sigma+\lambda$ and some algebraic computations that we omit.

On the other hand, if $x\in B\cap (B_{2\delta_{\mathcal{E}}}(E))^c\setminus \partial B_{2\delta_{\mathcal{E}}}(E)$, we know that $\psi_{B_\theta(E)}^{\delta_{\mathcal{E}}}=0$, and thus the bound for $\lVert D f_u^\eta(x)\rVert$ boils down to:
\begin{equation}
\begin{split}
\left\lVert \frac{D f_u^\eta(x)}{1-\eta}\right\rVert\leq&\lVert \text{id}_{\R^n}-D G(x)\rVert+\lVert D^2 f\rVert_\infty n^\frac{1}{2}\epsilon+\lVert D^3 f\rVert_\infty n\epsilon^2+\lVert D^2 f\rVert_\infty n^\frac{1}{2}\epsilon.
\nonumber
\end{split}
\end{equation}
In this case, we need a careful study of the quantity $\lVert \text{id}_{\R^n}-D G(x)\rVert$ since the bound yielded by Proposition \ref{ide}(ii) could be no longer true. On the other hand know that:
\begin{equation}
\begin{split}
\lVert D G(x)-\text{id}\rVert\leq\sup_{\substack{a\in\mathbb{S}^{n-1}\\b\in\mathbb{S}^{n-1}}}\sum_{i,j}^n \lvert b_i\rvert\lvert a_j\rvert \left\lvert\left\langle e_i, (D G(x)-\text{id})e_j\right\rangle\right\rvert,\nonumber
\end{split}
\end{equation}
and thus we are reduced to study the quantity $\lvert\left\langle e_i, (D G(x)-\text{id})e_j\right\rangle\rvert$.
If $i\neq j$, by Proposition \ref{alb}(iii) we have:
\begin{equation}
\left\lvert\left\langle e_i, (D G(x)-\text{id})e_j\right\rangle\right\rvert=\left\lvert\frac{\partial_jg_i(x)}{1+\beta(\sigma)}\right\rvert\leq\frac{\beta(\sigma)}{1+\beta(\sigma)}.\nonumber
\end{equation}
On the other hand if $i=j$, by definition of $G$, we have:
\begin{equation}
\lvert\left\langle e_i, (D G(x)-\text{id})e_i\right\rangle\rvert=\left\lvert\frac{\partial_ig_i(x)}{1+\beta(\sigma)}-1\right\rvert.\nonumber
\end{equation}
Therefore, the fact that $0\leq\partial_ig_i(x)\leq1$ (due to Proposition \ref{alb}(iv)) implies that $\lvert\left\langle e_i, (D G(x)-\text{id})e_i\right\rangle\rvert\leq 1$ and thus using Cauchy-Schwartz inequality, we conclude that:
\begin{equation}
\begin{split}
\lVert D G(x)-\text{id}\rVert\leq&\sup_{a\in\mathbb{S}^{n-1}}\sup_{b\in\mathbb{S}^{n-1}}\beta(\sigma)\sum_{i\neq j}^n \lvert b_i\rvert \lvert a_j\rvert +\sum_{i=1}^n\lvert a_i\rvert \lvert b_i\rvert\\
\leq& n^2 \beta(\sigma)+\sup_{a\in\mathbb{S}^{n-1}}\sup_{b\in\mathbb{S}^{n-1}} \left(\sum_{i=1}^n a_i^2\right)^\frac{1}{2}\left(\sum_{i=1}^n b_i^2\right)^\frac{1}{2}\leq n^2\beta(\sigma)+1.
\nonumber
\end{split}
\end{equation}
Summing up, our computations yield the bound:
\begin{equation}
\begin{split}
\left\lVert D f_u^\eta(x)\right\rVert\leq&(1-\eta)(1+n^2\beta(\sigma)+2\lVert D^2 f\rVert_\infty n^\frac{1}{2}\epsilon+\lVert D^3 f\rVert_\infty n\epsilon^2)<(1-\eta^2),
\nonumber
\end{split}
\end{equation}
where the last inequality comes from the condition on $\epsilon+\sigma+\lambda$ and some omitted algebraic computations.

We are left to prove point (iii). To do so, fix some $x\in E$ and let $h\in\left(-\theta/2,\theta/2\right)\setminus \{0\}$.
Since $f$ is smooth, we have the following Taylor expansions:

\begin{itemize}
    \item[(a)] $f(x+hv)=f(x)+D f(x)[hv]+\frac{D^2 f(x)[hv,hv]}{2}+r(x,hv)$,
    \item[(b)] $D f(x+hv)[\cdot]=D f(x)[\cdot]+D^2 f(x)[hv,\cdot]+R(x,hv)$,
    \item[(c)] $D^2 f(x+hv)[\cdot,\cdot]=D^2 f(x)[\cdot,\cdot]+D^3 f[hv,\cdot,\cdot]+\mathcal{R}(x,hv),$
\end{itemize}
where $\lvert r(x,hv)\rvert=O(\lvert h\rvert^3)$, $\lVert R(x,hv)\rVert= O(\lvert h\rvert^2)$ and $\lVert\mathcal{R}(x,hv)\rVert=O(\lvert h\rvert^2)$. Furthermore, for any $\rho>0$ the functions $\frac{r(\cdot,hv)}{h^2}$, $\frac{R(\cdot,hv)}{h}$ and $\frac{\mathcal{R}(\cdot,hv)}{h}$ converge uniformly in $x$ to $0$ on $B_\rho(0)$. To prove \eqref{numbero2001} we must study the difference quotient along $v$ of $F_u^\eta$ at points of $E$:
\begin{equation}
\begin{split}
&\frac{F_u^\eta(x+hv)-F_u^\eta(x)}{(1-\eta)h}=\underbrace{\frac{f(x+hv)-f(x)}{h}}_{\text{(I)}}
-\underbrace{\frac{D f(x+hv)[G(x+hv)]-D f(x)[G(x)]}{h}}_{\text{(II)}}\\
+&\underbrace{\frac{D^2 f(x+hv)[G(x+hv),G(x+hv)]-D^2 f(x)[G(x),G(x)]}{2h}}_{\text{(III)}}+\underbrace{\frac{(H_{v}(x+hv)-H_{v}(x))u}{h}}_{\text{(IV)}}.\label{abc}
\end{split}
\end{equation}
We are going to study each term of the right-hand side of the identity \eqref{abc} separately. 
Thanks to the smoothness of $f$ and its Taylor expansions, (I) turns into:
\begin{equation}
    \text{(I)}=\frac{f(x+hv)-f(x)}{h}=D f(x)[v]+\frac{D^2 f(x)[v,v]}{2}h+\frac{r(x,hv)}{h}.
    \label{numbero2003}
\end{equation}
Recall that $h\in\left(-\theta/2,\theta/2\right)\setminus \{0\}$ and thus $x,x+hv\in B_\theta(E)$. Hence thanks to Proposition \ref{alb} and th bound \eqref{smoo}, we have that:
\begin{equation}
\begin{split}
\text{(IV)}=H_{v}(x+hv)-H_{v}(x)=\frac{h}{1+\beta(\sigma)}.
\end{split}
\label{numbero2004}
\end{equation}
By Proposition \ref{ide} (ii), since $\lvert h\rvert<\theta/2$ we have that:
\begin{equation}
G(x+hv)=G(x)+hv+\Delta,
\label{GG}
\end{equation}
where $\lvert\Delta\rvert\leq 2(n+1)^\frac{1}{2}\frac{\beta(\sigma)\lvert h\rvert}{1+\beta(\sigma)}$.
Using the Taylor expansion (b) for $D f(x+hv)$ and the identity \eqref{GG}, we have:
\begin{equation}
\begin{split}
\text{(II)}=&\frac{D f(x)[G(x+hv)-G(x)]}{h}+D^2 f(x)[v, G(x+hv)]+\frac{R(x,hv)[G(x+hv)]}{h}\\
=&D f(x)[v]+\underbrace{\frac{D f(x)[\Delta]}{h}}_{\text{(IIa)}}+hD^2 f(x)[v, v]+D^2 f(x)[v,G(x)+\Delta]+\underbrace{\frac{R(x,hv)[G(x+hv)]}{h}}_{\text{(IIb)}}.
\label{numbero2005}
\end{split}
\end{equation}
Thanks to Proposition \ref{ide}, we have that:
\begin{equation}
   \lvert\text{(IIa)}\rvert\leq2(n+1)^\frac{1}{2}\beta(\sigma)\qquad\text{and}\qquad
\left\lVert\text{(IIb)}\right\rVert\leq\frac{\lVert R(x,hv)\rVert}{\lvert h\rvert}n^\frac{1}{2}\epsilon\label{numbero2006}.
\end{equation}
Eventually, (III) thanks to \eqref{GG} and the Taylor expansion (c) of $D^2 f(x+hv)$, becomes:
\begin{equation}
    \begin{split}
    \text{(III)}=\frac{1}{2}D^2 f(x)[v,v]h+&\underbrace{D^2 f(x)[v,G(x)+\Delta]}_{\text{(IIIa)}}+\underbrace{\frac{D^2 f(x)[G(x),\Delta]}{h}}_{\text{(IIIb)}}+\underbrace{\frac{D^2 f(x)[\Delta,\Delta]}{2h}}_{(IIIc)}\\
    +&\underbrace{\frac{D^3 f(x)[v,G(x+hv),G(x+hv)]}{2}}_{\text{(IIId)}}+\underbrace{\frac{\mathcal{R}(x,hv)[G(x+hv),G(x+hv)]}{2h}}_{\text{(IIIe)}}.
        \label{labelnove}
    \end{split}
\end{equation}
Again thanks to Proposition \ref{ide} we have the following estimates:
\begin{equation}
    \begin{split}
        &\lvert\text{(IIIa)}\rvert=\left\lvert D^2 f(x)[v,G(x)+\Delta]\right\rvert\leq \lVert D^2 f\rVert_\infty \left(n^\frac{1}{2}\epsilon+2(n+1)^\frac{1}{2}\beta(\sigma)\lvert h\rvert\right),\\
        &\lvert\text{(IIIb)}\rvert=\left\lvert D^2 f(x)[G(x),\Delta]\right\rvert/2\lvert h\rvert\leq(n+1)\lVert D^2 f(x)\rVert_\infty\beta(\sigma)\epsilon,\\
        &\lvert\text{(IIIc)}\rvert=\left\lvert D^2 f(x)[\Delta,\Delta]\right\rvert/2\lvert h\rvert\leq 2(n+1)\lVert D^2 f\rVert_\infty\beta(\sigma)^2\lvert h\rvert,\\
        &\lvert\text{(IIId)}\rvert=\left\lvert D^3 f(x)[v,G(x+hv),G(x+hv)]\right\rvert\leq 2\lVert D^3 f\rVert_\infty n\epsilon^2,\\
        &\lvert\text{(IIIe)}\rvert=\left\lvert\mathcal{R}(x,hv)[G(x+hv),G(x+hv)]\right\rvert\leq 2\left\lVert\mathcal{R}(x,hv)\right\rVert n\epsilon^2.
        \label{numbero4}
    \end{split}
\end{equation}
Thanks to the identities \eqref{numbero2003}, \eqref{numbero2004},\eqref{numbero2005} and \eqref{labelnove}, we have that \eqref{abc} simplifies to:
\begin{equation}
\begin{split}
\frac{F_u^\eta(x+hv)-F_u^\eta(x)}{(1-\eta)h}=&\frac{r(x,hv)}{h}-\frac{D f(x)[\Delta]}{h}-\frac{R(x,hv)[G(x+hv)]}{h}\\
+&\frac{D^3 f(x)[v,G(x+hv),G(x+hv)]}{2}+\frac{u}{1+\beta(\sigma)}\\
+&\frac{\mathcal{R}(x,hv)[G(x+hv),G(x+hv)]}{2h}
+\frac{D^2 f(x)[2G(x)+\Delta,\Delta]}{2h}.
\label{abc2}
\end{split}
\end{equation}
Furthermore, thanks to the estimates given in \eqref{numbero2006} and \eqref{numbero4} and identity \eqref{abc2} we deduce that:
\begin{equation}
\begin{split}
&\left\lvert\frac{F_u^\eta(x+hv)-F_u^\eta(x)}{h}-u \right\rvert\leq\frac{\lvert r(x,hv)\rvert}{\lvert h\rvert}+2(n+1)^\frac{1}{2}\beta(\sigma)+\frac{\lVert R(x,hv)\rVert}{\lvert h\rvert}n^\frac{1}{2}\epsilon\\
+&2(n+1)\lVert D^2 f\rVert_\infty(\beta(\sigma)^2\lvert h\rvert+\beta(\sigma)\epsilon)
+\frac{\left\lVert\mathcal{R}(x,hv)\right\rVert}{\lvert h\rvert} n\epsilon^2+\lvert \eta-\beta(\sigma)\rvert+\lVert D^3 f\rVert_\infty n\epsilon^2.
\label{label10}
\end{split}
\end{equation}
Since the convergence of remainders to $0$ of the Taylor expansions (a), (b) and (c) is uniform on $[0,1]^n$, we deduce that for any $k\in\N$ there exists $h^*\in(0,\min\{\eta,\theta/2\})$ such that for any $h\in(-2h^*,2h^*)\setminus\{0\}$:
$$\frac{\lvert r(x,hv)\rvert}{\lvert h\rvert}<\eta \qquad \text{and}\qquad\frac{\lVert R(x,hv)\rVert}{\lvert h\rvert}+\frac{\left\lVert\mathcal{R}(x,hv)\right\rVert}{\lvert h\rvert}\leq 1,$$
for any $x\in [0,1]^n$.
Moreover, thanks to \eqref{numbero2}, \eqref{label10} and some algebraic computations that we omit, we deduce that:
\begin{equation}
\begin{split}
\left\lvert\frac{F_u^\eta(x+hv)-F_u^\eta(x)}{h}-u \right\rvert
<8 \eta<10\eta- h^*.
\nonumber
\end{split}
\end{equation}
\end{proof}

We are ready to prove Proposition \ref{evev}, which we restate here for reader's convenience:

\evev*

\begin{proof}
For any $k_0<j<k$ we have $S_k^{v,w}\subseteq S_j^{v,w}$.
Indeed, if $g\in S_k^{v,w}$, there exist $0<\delta<1/k$ such that for any $x\in E$ one can find $\tau\in(-1/k,-\delta)\cup(\delta,1/k)\subset (-1/j,-\delta)\cup(\delta,1/j)$ for which:
\begin{equation}
\left\lvert\frac{g(x+\tau e)-g(x)-\tau w}{\tau}\right\rvert<\frac{1}{k}-\delta<\frac{1}{j}-\delta,
\nonumber
\end{equation}
which implies that $g\in S_j^{v,w}$.
To prove that $S_{k}^{v,w}$ is open, choose $\rho>0$ in such a way that:
$$\frac{2\rho}{\delta}+\left\lvert\frac{g(x+\tau e)-f(x)-\tau w}{\tau}\right\rvert<\frac{1}{k}-\delta,$$
for any $x\in E$. Then for any $f\in\lip([0,1]^n,\R^m)$ such that $\lVert g-f\rVert_\infty\leq\rho$ we have:
\begin{equation}
\left\lvert\frac{f(x+\tau e)-f(x)-\tau w}{\tau}\right\rvert<\frac{2\lVert g-f\rVert_\infty}{\delta}+\left\lvert\frac{g(x+\tau e)-f(x)-\tau w}{\tau }\right\rvert<\frac{1}{k}-\delta,
\nonumber
\end{equation}
and thus $g\in S_{k}^{v,w}$.
Since $\mathcal{C}^{\infty}(\R^n,\R^m)$ is dense in $\lip([0,1]^n,\R^m)$, to prove the density of $S^{v,w}_k$, we just need to show that we can approximate smooth functions with functions in $S_k^{v,w}$. Proposition \ref{lung} implies that for any $f\in\mathcal{C}^\infty(\R^n,\R^m)$ we can find a sequence of functions $\{g_i\}_{i\in\N}\subset\lip([0,1]^n,\R^m)$ for which:
\begin{itemize}
    \item[($\alpha$)] $\lVert f-g_i\rVert_\infty\leq (i+k)^{-1}\lVert f\rVert_\infty$,
    \item[($\beta$)] there exist $0<h^*<(10(i+k))^{-1}$ such that for any $x\in E$ there is a $t\in (-(i+k)^{-1},-h^*)\cup(h^*,(i+k)^{-1})$ such that:
    $$\left\lvert\frac{g_i(x+te)-g_i(x)-wt}{t}\right\rvert<\frac{1}{i}-h^*.$$
\end{itemize}
Point ($\beta$) above implies that $g_i\in S_{i+k}^{v,w}$ for any $i\in\N$ and thus $\{g_i\}\subseteq S_k^{v,w}$. On the hand, point ($\alpha$) implies that $g_i$ are uniformly converging to $f$, proving the claim.
\end{proof}

\section{Construction of the winning strategy for Player II}
\label{Ss4}
Given an open interval $I\subseteq \R$, and a measurable function $f:I\to\R$, we recall that the Hardy-Littlewood maximal function $Mf$ of $f$ is defined as:
\begin{equation}
M f(t):=\sup_{\substack{t\in J\subseteq I\\J\text{ open}}}\fint_J\lvert f(t)\rvert dt.\nonumber
\end{equation}
It is well known that $M$ is a bounded operator on $L^p$ for any $p\in (1,\infty]$ and:
\begin{equation}
\lVert Mf\rVert_{L^p(I)}\leq \left(5 p\frac{2^{p-1}}{p-1}\right)^\frac{1}{p}\lVert f\rVert_{L^p(I)}\label{me}.
\end{equation}
For a proof of the inequality \eqref{me}, see for instance Section 3 of Chapter 1  in \cite{stein}.
\begin{definizione}
Let $I$ be an open interval and $\gamma:I\to \R^n$ the canonical parametrization of a $C(e,\sigma)$-curve. For any $f:\R^n\to\R^m$ Lipschitz function we define:
\begin{equation}
M_\gamma f:=M[(f\circ \gamma)^\prime],
\nonumber
\end{equation}
where $(f\circ\gamma)^\prime$ is the derivative of the Lipschitz curve $f\circ \gamma$.
\end{definizione}

The following proposition shows that if $f,g\in\mathfrak{P}(n,m)$ and $\gamma$ is a $C(e,\sigma)$-curve, up to a small error depending only on $\sigma$, one can estimate the $L^p$ norm of $M_\gamma(f-g)$ with $\lVert f-g\rVert_\infty$. This bound will be used in Proposition \ref{lemma1} to get an estimate of the width of the set where the derivative of $(f-g)$ is big along the direction $e$.

\begin{proposizione}\label{princ}
Let $e\in\mathbb{S}^{n-1}$, $\sigma\in(0,1)$, $p\in(4,\infty)$ and $\gamma:I\to [0,1]^n$, the canonical parametrization of a piece-wise affine $C(e,\sigma)$-curve. Let $f\in\mathfrak{P}(n,m)$ and suppose that $\Pi_f\in\tau$ is a partition of $[0,1]^n$ adapted to $f$.
Then, for any $g\in\mathfrak{P}(n,m)$ we have that:
\begin{equation}
\begin{split}
\lVert M_\gamma (f-g)\rVert_{L^p(I)}^p\leq 8^p\text{Card}(\Pi_f)^{p/2}n(\lVert f-g\rVert^{1/2}_{\infty}+\beta(\sigma)^{p/2}).
\label{numbero101}
\end{split}
\end{equation}
\end{proposizione}

\begin{proof}
Let $J\subseteq I$ be an open interval and for any $P\in\Pi_f$ define:
\begin{equation}
T_P:=\{s\in J: \gamma(s)\in\text{cl}(P)\}.\nonumber
\end{equation}
Since $\gamma$ is the canonical parametrization of a piece-wise affine curve, there are finitely many disjoint open intervals $I_k$ for which $\gamma\vert_{I_k}$ is a non-degenerate segment and $J\subseteq \bigcup_{k=1}^{N_1} \text{cl}(I_k)$. This implies that $T_P$ has a finite number  of connected components since the $P$'s are convex. Furthermore, since $\gamma$ is continuous, $(T_P)^c$ is relatively open in $J$.
Define $a(P):=\inf T_P$ and $b(P):=\sup T_P$ and let: $$(a(P),b(P))\cap(T_P)^c=\bigcup_{j=1}^{N_2} (a_j,b_j).$$
Define on $J_P:=(a(P),b(P))$ the curve $\gamma_P:J_P\to \R^n$, which is the linear interpolation of $\gamma$ inside $P$:
\begin{equation}
\gamma_P(t):=
\begin{cases}
\gamma(t) &\text{ for } t\in T_P,\\
\gamma(a_j)+\frac{\gamma(b_j)-\gamma(a_j)}{b_j-a_j}(t-a_j) &\text{ for } t\in (a_j,b_j).
\end{cases}
\nonumber
\end{equation}
Note that $\gamma_P$ is a $C(e,\sigma)$-curve which by construction is contained in $\text{cl}(P)$ and coincides with $\gamma$ on $T_P$. Moreover one can check that $\gamma_P^\prime(t)=e+\eta_P^\prime(t)$, where $\eta^\prime_P\in e^\perp$ and that $\lvert \eta_P^\prime(t)\rvert\leq \beta(\sigma)$ for almost every $t\in J_P$. Define $\mathfrak{d}:=f-g$, and note that we have that the following inequality holds:
\begin{equation}
\fint_J\left\lvert (\mathfrak{d}\circ\gamma)^\prime(t)\right\rvert^2 dt\leq\frac{1}{\mathcal{L}^1(J)}\sum_{P\in\Pi_f} \int_{J_P}\left\lvert (\mathfrak{d}\circ\gamma_P)^\prime(t)\right\rvert^2 dt\label{bd1}.
\end{equation}
Since $g$ is piece-wise congruent mapping, for any $P$ there is a finite family of open disjoint subintervals $\{(r_i,r_{i+1})\}_{i=1,\ldots,N_3}$ of $J_P$ such that:
\begin{itemize}
\item[(i)] For any $i=1,\ldots,N_3$ there are $A,B_i\in O(n,m)$ and $a,b_i\in\R^m$ for which:
$$\mathfrak{d}(\gamma_P(t))=(A-B_i)\gamma_P(t)+(a-b_i),$$
for any $t\in (r_i,r_{i+1})$,
\item[(ii)] $J_P\subseteq\bigcup_{i=1}^{N_3} [r_i,r_{i+1}]$.
\end{itemize}
Points (i) and (ii) above imply that the right-hand side of \eqref{bd1} can be rewritten as: 
\begin{equation}
\begin{split}
&\int_{J_P}\left\lvert (\mathfrak{d}\circ\gamma_P)^\prime(t)\right\rvert^2 dt=\sum_{i=1}^{N_3} \int_{r_i}^{r_{i+1}}\left\lvert (A-B_i)\gamma^\prime_P(t)\right\rvert^2 dt\\
=&\sum_{i=1}^{N_3} \int_{r_i}^{r_{i+1}}\big(\lvert A\gamma_P^\prime(t)\rvert^2+\lvert B_i\gamma_P^\prime(t)\rvert^2-2\langle A\gamma_P^\prime(t),B_i\gamma_P^\prime(t)\rangle\big)dt\\
=2\sum_{i=1}^{N_3} \int_{r_i}^{r_{i+1}}\big(&\lvert A\gamma_P^\prime(t)\rvert^2-\langle A\gamma_P^\prime(t),B_i\gamma_P^\prime(t)\rangle\big)dt=2\sum_{i=1}^{N_3} \int_{r_i}^{r_{i+1}}\left\langle A\gamma_P^\prime(t),(A-B_i)\gamma_P^\prime(t) \right\rangle dt,
\label{numbero12}
\end{split}
\end{equation}
where the second last equality comes from the fact that $A$ and $B_i$ are both distance preserving and thus $\lvert Av\rvert=\lvert B_iv\rvert=\lvert v\rvert$ for any $v\in\R^n$.
Furthermore, since $\gamma^\prime_P=e+\eta^\prime_p$ and:
$$\lvert\left\langle A\eta_P^\prime(t),(A-B_i)\gamma_P^\prime(t) \right\rangle\rvert\leq \lvert\eta_P^\prime\rvert\cdot2\lvert\gamma_P^\prime(t)\rvert\leq2\beta(\sigma)(1+\beta(\sigma)),$$
we have that for any $i\in{1,\ldots, N_3}$:
\begin{equation}
\begin{split}
\int_{r_i}^{r_{i+1}}\left\langle A\gamma_P^\prime(t),(A-B_i)\gamma_P^\prime(t) \right\rangle dt
\leq\int_{r_i}^{r_{i+1}}\langle Ae,&(A-B_i)\gamma^\prime_P(t) \rangle dt+2(r^{i+1}-r^i)\beta(\sigma)\left(1+\beta(\sigma)\right).
\label{numbero11}
\end{split}
\end{equation}
Furthermore, the first term of the right-hand side of \eqref{numbero11} thanks to the fundamental theorem of calculus can be rewritten as:
\begin{equation}
\begin{split}
&\int_{r_i}^{r_{i+1}}\left\langle Ae,(A-B_i)(\gamma_P^\prime(t)) \right\rangle dt=\left\langle Ae,\mathfrak{d}(\gamma_P(r_{i+1}))-\mathfrak{d}(\gamma_P(r_i)) \right\rangle.
\label{numbero10}
\end{split}
\end{equation}
Thus, putting together \eqref{numbero12}, \eqref{numbero11} and \eqref{numbero10}, we deduce that:
\begin{equation}
\begin{split}
\int_{J_P}\left\lvert (\mathfrak{d}\circ\gamma_P)^\prime(t)\right\rvert^2 dt
\leq 2\big\langle Ae,\mathfrak{d}(\gamma_P(b(P)))-\mathfrak{d}(\gamma_P(a(P))\big\rangle
+4(b(P)-a(P))\beta(\sigma)\left(1+\beta(\sigma)\right),
\label{numbero21}
\end{split}
\end{equation}
which in the case $\eta^\prime_P=0$, turns into:
\begin{equation}
\begin{split}
\int_{J_P}\left\lvert (\mathfrak{d}\circ\gamma_P)^\prime(t)\right\rvert^2 dt
=&2 \big\langle Ae,\mathfrak{d}(\gamma_P(b(P)))-\mathfrak{d}(\gamma_P(a(P)))\big\rangle \leq 2\lVert \mathfrak{d}\rVert_\infty=2\lVert f-g\rVert_\infty.
\end{split}\label{eqeqeq}
\end{equation}
Defined $v_P:=\frac{\gamma(b(P))-\gamma(a(P))}{b(P)-a(P)}-e$, we let $\tilde{\gamma}:(a(P),b(P))\to\R^n$ be the curve: $$\tilde{\gamma}(t):=(e+v_P)(t-a(P)).$$
The first identity in \eqref{eqeqeq} and the fact that $\lvert v_P\rvert\leq \beta(\sigma)$ imply that:
\begin{equation}
\begin{split}
\Big\langle Ae,\mathfrak{d}(\gamma_P(b(P)))-\mathfrak{d}(\gamma_P(a(P))) \Big\rangle
\leq&\Big\langle A\tilde{\gamma}_P^\prime,\mathfrak{d}(\gamma_P(b(P)))-\mathfrak{d}(\gamma_P(a(P)))  \Big\rangle+2\beta(\sigma)(b(P)-a(P))\\
=&\frac{1}{2}\int_{J_P}\lvert (\mathfrak{d}\circ\tilde{\gamma}_P)^\prime(t)\rvert^2dt +2\beta(\sigma)(b(P)-a(P)).
\label{numbero20}
\end{split}
\end{equation}
Putting together \eqref{numbero21} and \eqref{numbero20}, we deduce that:
\begin{equation}
\begin{split}
\int_{J_P}\left\lvert (\mathfrak{d}\circ\gamma_P)^\prime(t)\right\rvert^2 dt
\leq&\int_{J_P}\lvert (\mathfrak{d}\circ\tilde{\gamma}_P)^\prime(t)\rvert^2dt+4(b(P)-a(P))\beta(\sigma)\left(2+\beta(\sigma)\right).
\label{bd2}
\end{split}
\end{equation}
Thus, the bounds \eqref{bd1}, \eqref{bd2} and the fact that $\beta(\sigma)\leq1$, which is due to the fact that $\sigma<1/10$, yield:
\begin{equation}
\begin{split}
\fint_J\left\lvert (\mathfrak{d}\circ\gamma)^\prime\right\rvert^2 dt
\leq\fint_J\sum_{P\in\Pi_f}\chi_{J_P}(t)\left\lvert (\mathfrak{d}\circ\tilde{\gamma}_P)^\prime(t)\right\rvert^2 dt+16\beta(\sigma)\text{Card}(\Pi_f).\label{label11}
\end{split}
\end{equation}
Now that we have an estimate for $\fint_J\left\lvert (\mathfrak{d}\circ\gamma)^\prime\right\rvert^2 dt
$ we will use \eqref{me} and Jensen's inequality to prove that \eqref{numbero101} holds. Thanks to Jensen's inequality we have:
\begin{equation}
\begin{split}
M_\gamma \mathfrak{d}(t)\leq\sup_{\substack{t\in J\subseteq I\\J\text{ open}}}\left(\fint_J\left\lvert (\mathfrak{d}\circ\gamma)^\prime(t)\right\rvert^2 dt\right)^\frac{1}{2}.
\label{label12}
\end{split}
\end{equation}
Combining the bounds given by \eqref{label11} and \eqref{label12}, we deduce that:
\begin{equation}
M_\gamma \mathfrak{d}(r)\leq\sup_{\substack{r\in J\subseteq I\\J\text{ open}}}\left(\fint_J\sum_{P\in\Pi_f}\chi_{J_P}(t)\left\lvert h_P(t)\right\rvert^2 dt+16\beta(\sigma)\text{Card}(\Pi_f)\right)^{1/2},
\label{bd3}
\end{equation}
where $h_P(t):=(\mathfrak{d}\circ\tilde{\gamma}_P)^\prime(t)$. Furthermore, Jensen's inequality and \eqref{bd3} imply: 
\begin{equation}
\begin{split}
\lVert M_\gamma \mathfrak{d}\rVert_{L^p(I)}^p
\leq2^{p/2-1}&\int_{I}\bigg(\sup_{\substack{r\in J\subseteq I\\J\text{ open}}}\fint_J\sum_{P\in\Pi_f}\chi_{J_P}(t)\left\lvert h_P(t)\right\rvert^2 dt\bigg)^{p/2}dr+\frac{\mathcal{L}^1(I)(32\beta(\sigma)\text{Card}(\Pi_f))^{p/2}}{2}.
\label{bd4}
\end{split}
\end{equation}
Applying the strong $L^p$ estimate \eqref{me} of the Maximal operator on the real line to the first term on the right-hand side of \eqref{bd4}, we obtain:
\begin{equation}
\begin{split}
\lVert M_\gamma \mathfrak{d}\rVert_{L^p(I)}^p
\leq&2^{p+3}\int_{I}\bigg(\sum_{P\in\Pi_f}\chi_{J_P}(t)\left\lvert h_P(t)\right\rvert^2 \bigg)^{p/2}dt+\frac{\mathcal{L}^1(I)(32\beta(\sigma)\text{Card}(\Pi_f))^{p/2}}{2}.
\nonumber
\end{split}
\end{equation}
Using the interpolation H\"oelder's inequality $\lVert f\rVert_{L^p}^p\leq\lVert f\rVert_{L^2}\lVert f\rVert_{L^{2(p-1)}}^{p-1}$, Jensen's inequality and the fact that $\mathfrak{d}\circ\gamma$ is $2(1+\beta(\sigma))$-Lipschitz, we deduce that:
\begin{equation}
\begin{split}
\int_{I}\Big(\sum_{P\in\Pi_f}&\chi_{J_P}(t)\left\lvert h_P(t)\right\rvert^2 \Big)^{p/2}dt
\leq\text{Card}(\Pi_f)^{{p/2}-1}\int_{I}\sum_{P\in\Pi_f}\chi_{J_P}(t)\left\lvert h_P(t)\right\rvert^p dt\\
\leq&\text{Card}(\Pi_f)^{p/2-1}\sum_{P\in\Pi_f}\left(\int_{J_P}\left\lvert h_P(t)\right\rvert^2 dt\right)^{1/2}\left(\int_{J_P}\left\lvert h_P(t)\right\rvert^{2(p-1)} dt\right)^{1/2}\\
\leq&4^{p-1}\text{Card}(\Pi_f)^{p/2-1}\sum_{P\in\Pi_f}\mathcal{L}^1(J_P)^{1/2}\left(\int_{J_P}\left\lvert h_P(t)\right\rvert^2 dt\right)^{1/2},
\end{split}
\nonumber
\end{equation}
Thanks to the inequality \eqref{eqeqeq}, we have $\int_{J_P}\left\lvert h_p(t)\right\rvert^2 dt\leq 2\lVert f-g\rVert_\infty$ and thus:
\begin{equation}
\begin{split}
\int_{I}\Big(\sum_{P\in\Pi_f}&\chi_{J_P}(t)\left\lvert h_P(t)\right\rvert^2 \Big)^{p/2}dt\leq 4^p\text{Card}(\Pi_f)^{p/2}\mathcal{L}^1(I)^{1/2}\lVert f-g\rVert_\infty^{1/2}.
\end{split}
\label{label13}
\end{equation}
Thanks to Proposition \ref{diam} and the inequalities \eqref{bd4} and \eqref{label13}, we conclude that:
\begin{equation}
\begin{split}
\lVert M_\gamma \mathfrak{d}\rVert_{L^p(I)}^p\leq 8^p\text{Card}(\Pi_f)^{p/2}n(\lVert f-g\rVert^{1/2}_\infty+\beta(\sigma)^{p/2}),
\nonumber
\end{split}
\end{equation}
which concludes the proof.
\end{proof}

The following lemma, stated here without proof, is an immediate consequence of the properties of Lipschitz functions.

\begin{lemma}\label{pippol}
Let $0<\epsilon$ and $F:\R^n\to\R^m$ be a $2$-Lipschitz function. Suppose there are $x,y\in\R^n$ such that:
\begin{equation}
\lvert F(y)-F(x)\rvert\geq\epsilon\lvert x-y\rvert\nonumber.
\end{equation}
Then, for any $0<r<\frac{\epsilon}{16}\lvert x-y\rvert$, $z\in B_r(x)$ and $w\in B_r(y)$ we have:
\begin{equation}
\lvert F(w)-F(z)\rvert\geq\frac{\epsilon}{2}\lvert z-w\rvert\nonumber.
\end{equation}
\end{lemma}

The following theorem shows that if two piece-wise congruent functions are close in the supremum norm, the set where their directional derivatives along a fixed direction are far away has small width.

\begin{teorema}\label{lemma1}
Let $0<\epsilon,\omega<1/10nk_0$, $k_0\in\N$ and $e\in\mathbb{S}^{n-1}$. Let $f\in\mathfrak{P}(n,m)$, $\Pi_f\in\tau$ be a partition of $[0,1]^n$ adapted to $f$ and $\sigma\in(0,2^{-10}\text{Card}(\Pi_f)^{-1}\epsilon^{4})$. Then there exists a neighbourhood $V$ of $f$ in $\lip([0,1]^n,\R^m)$ for which:
\begin{itemize}
\item[(i)] $\text{diam}(V)< 2^{-8p}n^{-2}\text{Card}(\Pi_f)^{-p}\epsilon^{2p}\omega^2$ where $p:=-\text{log}(\omega)/\text{log}(2)$,
\item[(ii)] for any $g\in V\cap\mathfrak{P}(n,m)$ there exist open set $G\subseteq (0,1)^n$ such that:
\begin{itemize}
\item[(a)] $w_{C(e,\sigma)}[(4/k_0,1-4/k_0)^n\setminus G]<(1+n^\frac{1}{2})\omega$,
\item[(b)] $\lvert f(x+te)-f(x)-g(x+te)+g(x)\rvert<\epsilon \lvert t\rvert$ for any $x\in G\cap (4/k_0,1-4/k_0)^n$ and any $t\in\R$ such that $x+te\in[1/2k_0,1-1/2k_0]^n$.
\end{itemize}
\end{itemize}
\end{teorema}

\begin{proof}
For any $g\in\mathfrak{P}(n,m)$ we define:
\begin{itemize}
    \item[(i)] $F(g)$ as the set of those $x\in[1/2k_0,1-/2k_0]^n$ for which there exists $t=t(x)\in\R$ such that $x+te\in[1/2k_0,1-1/2k_0]^n$ and:
    $$\lvert f(x+te)-f(x)-g(x+te)+g(x)\rvert\geq\epsilon \lvert t\rvert.$$
Note that by the Lipschitzianity of $f$ and $g$ and the compactness of $[1/2k_0,1-1/2k_0]^n$ imply that $F(g)$ is compact.
\item[(ii)] For any $x\in F(g)$ we let $T(x)$ to be the compact set of those $t\in\R$ for which
\begin{equation}
\lvert f(x+te)-f(x)-g(x+te)+g(x)\rvert\geq\epsilon \lvert t\rvert,\nonumber
\end{equation}
and $x+te\in[1/2k_0,1-1/2k_0]^n$. Furthermore we define for any $x\in F(g)$
\begin{equation}
r(x):=\begin{cases}
\min T(x) &\text{ if } \lvert\min T(x)\rvert\geq  \lvert\max T(x)\rvert,\\
\max T(x) &\text{ otherwise}.
\end{cases}
\nonumber
\end{equation}
\item[(iii)]Finally, we introduce the following neighbourhood of $E$:
\begin{equation}
G_\epsilon:=\bigcup_{x\in F_2}B_{\frac{\epsilon\lvert r(x)\rvert}{64}}(x).\nonumber
\end{equation}
\end{itemize}

Let $\gamma:(0,T)\to\R^n$ be a curve in $\Gamma_{e,\sigma}$, which was introduced in Definition \ref{spessore}, and assume that there exists $\overline{t}\in (0,T)$ such that $\gamma(\overline{t})\in G_\epsilon$.
This implies that there is $x\in F(g)$ such that $\lvert x-\gamma(\overline{t})\rvert\leq \frac{\epsilon}{64}\lvert r(x)\rvert$.
We define now the curve $\tilde{\gamma}:\R\to\R^n$ as $\tilde{\gamma}(t):=te+\tilde{\eta}(t)$, where:
\begin{equation}
    \tilde{\eta}(t):=\begin{cases}
\eta(0) & \text{if } t\leq0,\\
\eta(t) &\text{if } t\in(0,T)\\
\eta(T) &\text{if } t\geq T.
\end{cases}
\nonumber
\end{equation}
The curve $\tilde{\eta}$ is $\beta(\sigma)$-Lipschitz since $\eta$ is and this implies that:
\begin{equation}
\begin{split}
    \lvert \tilde{\gamma}(\overline{t}+r(x))-x-r(x)e\rvert=&\lvert (\overline{t}+r(x))e+\tilde{\eta}(\overline{t}+r(x))-x-r(x)e\rvert=\lvert \overline{t}e+\tilde{\eta}(\overline{t}+r(x))-x\rvert\\
\leq & \lvert \tilde{\gamma}(\overline{t})-x\rvert+\lvert\tilde{\eta}(\overline{t})-\tilde{\eta}(\overline{t}+r(x))\rvert\leq \Big(\frac{\epsilon}{64}+\beta(\sigma)\Big)\lvert r(x)\rvert\leq \frac{3\epsilon\lvert r(x)\rvert}{64},
\end{split}
    \label{numbero1002}
\end{equation}
where the last inequality comes from the fact that $\beta(\sigma)\leq \epsilon/32$, thanks to the choice of $\sigma$. In particular since $\lvert r(x)\rvert\leq n$, we have:
\begin{equation}
    \lvert \tilde{\gamma}(\overline{t}+r(x))-x-r(x)e\rvert\leq \frac{n\epsilon}{16}\leq \frac{1}{160k_0}.
    \label{numbero1001}
\end{equation}
Since by definition of $r(x)$ we have that $1/2k_0\leq\text{dist}(([0,1]^n)^c,x+r(x)e)$, the bound \eqref{numbero1001} implies that $\tilde{\gamma}(\overline{t}+r(x))\in (0,1)^n$. We recall that since $\gamma\in\Gamma_{e,\sigma}$, then $\tilde{\gamma(t)}\not\in(0,1)^n$ if and only if $T\not\in(0,T)$. 
Therefore, we have that $\overline{t}+r(x)\in(0,T)$.
Thanks to the bound \eqref{numbero1002}, the fact that $\lvert \gamma(\overline{t})-x\rvert\leq \epsilon\lvert r(x)\rvert/64$ and that $x\in F(g)$ we have by Lemma \ref{pippol} that:
\begin{equation}
\frac{\epsilon}{2} \lvert r(x)\rvert\leq\lvert f(\gamma(\overline{t}+r(x)))-f(\gamma(\overline{t}))-g(\gamma(\overline{t}+r(x)))+g(\gamma(\overline{t}))\rvert.
\nonumber
\end{equation}
Thanks to the fundamental theorem of calculus and the above inequality, we deduce that:
\begin{equation}
\begin{split}
\frac{\epsilon}{2}\leq&\left\lvert\fint_{\overline{t}}^{\overline{t}+r(x)} \left\lvert ((f-g)\circ\gamma)^\prime(s)\right\rvert ds\right\rvert\leq M_\gamma (f-g)(\overline{t}),\label{label14}
\end{split}
\end{equation}
for any $\overline{t}$ for which $\gamma(\overline{t})\in G_\epsilon$. Inequality \eqref{label14} together with Proposition \ref{canone} imply:
\begin{equation}
\begin{split}
\frac{\epsilon^p}{2^p}w_{C(e,\sigma)}[F(g)]\leq&\inf_{\substack{F(g)\subseteq G\subseteq G_\epsilon\\ G\text{ open}}}\sup_{\gamma\in\Gamma_{e,\sigma}}\int_{\substack{t\in I\\ \gamma(t)\in G}} M_\gamma (f-g)(t)^p \lvert\gamma^\prime(t)\rvert dt\\
\leq&\left( 1+\beta(\sigma)\right)\sup_{\gamma\in\Gamma_{e,\sigma}}\int_{I} M_\gamma (f-g)(t)^pdt.
\end{split}\label{eql}
\end{equation}
Finally Proposition \ref{princ} and \eqref{eql} imply:
\begin{equation}
\begin{split}
w_{C(e,\sigma)}[F(g)]\leq 16^p\epsilon^{-p}\text{Card}(\Pi_f)^{p/2}n(\lVert f-g\rVert^{1/2}+\beta(\sigma)^{p/2}).
\end{split}
\nonumber
\end{equation}
In order to prove the proposition we just need to estimate the two terms of the right-hand side in the above inequality. By hypothesis we know that $\lVert f-g\rVert_\infty<n^{-2} 2^{-8p}\text{Card}(\Pi_f)^{-p}\epsilon^{2p}\omega^2$ with $p=-\log(\omega)/\log(2)$. Therefore the first term can be estimated by:
\begin{equation}
16^p\epsilon^{-p}\text{Card}(\Pi_f)^{p/2}n\lVert f-g\rVert_\infty^{1/2}<\omega.
\nonumber
\end{equation}
On the other hand, since $\beta(\sigma)\leq 4\sigma^{1/2}$ since $\sigma <1/10$, we have that:
\begin{equation}
\begin{split}
16^p\epsilon^{-p}\text{Card}(\Pi_f)^{p/2}n\beta(\sigma)^{p/2}\leq& 32^p\epsilon^{-p}\text{Card}(\Pi_f)^{p/2}n\sigma^{p/2}<n\epsilon^p,
\nonumber
\end{split}
\end{equation}
where the first inequality comes from the fact that $\sigma<2^{-10}\text{Card}(\Pi_f)^{-1}\epsilon^{4}$. Since $\epsilon<1/10$, then $\epsilon^p<2^{-p}=\omega$ and this proves the Proposition with the choice $G:=(0,1)^n\setminus F(g)$.
\end{proof}

As mentioned in the introduction the proof of the implication (ii)$\Rightarrow$(i) of Theorem \ref{main}, will be proved by means of a topological game, that we formally introduce here.

\begin{definizione}
Let $(X,\mathcal{T})$ be a topological space. The Banach-Mazur game, is a game between two players, Player I and Player II, where
Player I is dealt with an arbitrary subset $A\subseteq X$ and Player II with the set $B:=X\setminus A$.

The game $\langle A,B\rangle$ is played as follows: I chooses arbitrarily an open set $U_1\subseteq X$; then II chooses an open set $V_1\subseteq U_1$; then I chooses an open set $U_2\subseteq V_1$ and so on.
If the set $\left(\bigcap_{n\in\N} V_i\right)\cap A\neq \emptyset$ then I wins. Otherwise II wins.\label{BMgame}
\end{definizione}

The following proposition relates the Banach-Mazur game to the topology of the space on which it is played.

\begin{teorema}
There exists a strategy by which Player II can be sure to win if and only if $B$ is residual in $X$, or equivalently if and only if $A$ is meagre.\label{BMgioco}
\end{teorema}

\begin{proof}
The proof of this result is given in~\cite{Oxtobi} only in the case of the real line. However that argument works in the same way in any complete metric space.
\end{proof}

\ueue*

\begin{proof}

Let $E_1\supseteq E_2\supseteq\ldots$ be relarively dense open subset of $F$ such that $\bigcap_{k=1}^\infty E_k\subseteq E$. 
To prove that $S$ is residual we will build a winning strategy for the Player II in the Banach-Mazur game where Player II is dealt with the set:
$$B:=\{f\in\lip([0,1]^n,\R^m): f\text{ is }not\text{ fully non-differentiable on }E\cap F\}.$$
In addition to the non-empty open subsets $V_k\subseteq \lip([0,1]^n,\R^m)$, we will make the second player choose: closed subsets $\mathcal{A}_k$ of $F$, functions in $V_k\cap\mathfrak{P}(n,m)$, directions $e_k\in\mathbb{S}^{n-1}$, $\sigma_k>0$ and non-empty relatively open subsets $M_k$ in $\mathcal{A}_k$ in such a way that:
\begin{itemize}
\item[(i)] $\diam(V_k)\leq 1/2^k$, $\sigma_k\leq 1/2^k$ and $\mathcal{A}_k$ has every portion of positive $\omega_{C(e_k,\sigma_k)}$-width,
\item[(ii)] $e_k\in C(e_{k-1},\sigma_{k-1})$ for any $k\geq 2$,
\item[(iii)] $\mathcal{A}_k\subseteq \mathcal{A}_{k-1}$ for any $k\geq 2$,
\item[(iv)]for every $g\in V_k\cap\mathfrak{P}(n,m)$ there is an open set $G\subseteq(0,1)^n$ such that:
\begin{itemize}
\item[(a)
] $w_{C(e_k,\sigma_k)}[(1/k_0,1-1/k_0)^n\setminus G]<w_{C(e_k,\sigma_k)}[M_k\cap E_k]$,
\item[(b)] $\lvert g(x+te_k)-g(x)-f_k(x+te_k)+f_k(x)\rvert\leq 2^{-k}\lvert t\rvert$ for any $x\in G\cap(1/k_0,1-1/k_0)^n$ and any $t\in\R$ for which $x+te_k\in [1/2k_0,1-1/2k_0]^n$,
\item[(c)] $f_k$ and $g$ have $2^{-k}$-directional derivative along $e_k$ for any $x\in G$,
\end{itemize}
\item[(v)] For any $x\in M_k$, $f_i$ are $2^{-i}$-directionally derivable along $e_i$ for any $i=1,\ldots, k$,
\item[(vi)]$\text{cl}(M_k)\subseteq M_{k-1}\cap E_{k-1}\cap\mathcal{A}_k$ if $k\geq 2$,
\item[(vii)]$\lvert f_k(x+te_{k-1})-f_k(x)-(f_{k-1}(x+te_{k-1})-f_{k-1}(x))\rvert\leq 2^{-k+1}\lvert t\rvert$ for any $x\in M_k$ and any $t\in\R$ for which $x+te_{k-1}\in [1/2k_0,1-1/2k_0]^n$ for any $k\geq 2$.
\end{itemize}

\paragraph*{Construction of the first move of Player II}

The construction of the answer of Player II to the first move $U_1$ of Player I starts by picking an arbitrary function $f_1\in U_1\cap\mathfrak{P}(n,m)$, which exists by Corollary \ref{dense}. Let $\tilde{M}_1$ be the set on which $f_1$ has at least one directional derivative on $F$ and note that $\tilde{M}_1$ is relatively open since $f_1$ is a piece-wise congruent mapping. Let $\Pi_{f_1}$ be a partition adapted to $f_1$, $8\sigma_1:=2^{-100(n+1)}\text{Card}(\Pi_{f_1})^{-2}(nk_0)^{-1}$ and let $\{u_i\}_{i=1,\ldots, N_1}$ be a finite $\sigma_1/16$-dense set in $\mathbb{S}^{n-1}$. Using Proposition \ref{splitter} we know that there exists $i\in\{1,\ldots, N_1\}$ and a closed subset $\mathcal{A}_1$ of $F$ which has every portion of positive $C(u_i,\sigma_1/8)$-width. We therefore define $e_1:=u_i$, $\epsilon:=2^{-10(n+1)}(nk_0)^{-1}$ and:
$$\omega_1:=\min\{2^{-10(n+1)},w_{C(e_1,\sigma_1)}(\tilde{M}_1\cap E_1\cap \mathcal{A}_1)/(1+n^{1/2})\}.$$
Applying Proposition \ref{lemma1} to $f_1$, $\omega_1$, $e_1$, $\sigma_1$ and $\epsilon$ we get an open neighbourhood $V_1$ of $f_1$ (with $\text{diam}(V_1)<2^{-1}$) such that for every $g\in V_1\cap \mathfrak{P}(n,m)$ there exist an open set $\tilde{G}$ for which:
\begin{itemize}
\item[(a')] $w_{C(e_1,\sigma_1)}[(1/k_0,1-1/k_0)^n\setminus \tilde{G}]< w_{C(e_1,\sigma_1)}[\tilde{M}_1\cap E_1\cap \mathcal{A}_1]$.
\item[(b')]$\lvert g(x+te_1)-g(x)-(f_1(x+te_1)-f_1(x))\rvert\leq\epsilon\lvert t\rvert\leq 2^{-10(n+1)}\lvert t\rvert\leq\lvert t\rvert/2$ for any $x\in \tilde{G}\cap (1/k_0,1-1/k_0)^n$ and any $t\in \R$ such that $x+te_1\in[1/2k_0,1-1/2k_0]^n$.
\end{itemize}
Moreover since $\sigma_1<\epsilon_1^2/2$, applying Proposition \ref{fortuna} to $f_1$, $g$, $e_1$, $\sigma_1$ and $\epsilon_1$ we obtain that $\text{cl}(\Xi(f_1,e_1,\epsilon_1))$ and $\text{cl}(\Xi(g,e_1,\epsilon_k))$ are $C(e,\sigma)$-null. We define:
\begin{itemize}
\item $M_1:=\tilde{M}_1\cap \mathcal{A}_1\setminus\text{cl}(\Xi(f_1,e_1,\epsilon_1))$,
\item $G:=\tilde{G}\setminus(\text{cl}(\Xi(f_1,e_1,\epsilon_1))\cup \text{cl}(\Xi(g,e_1,\epsilon_1)))$,
\end{itemize}
and note that $M_1$ si relatively open in $\mathcal{A}_1$ and $G$ is open. Thus $G$ is the open associated to $g$ that satisfies to our requirements:
\begin{itemize}
\item[(a)] $w_{C(e_1,\sigma_1)}[(1/k_0,1-1/k_0)^n\setminus G]< w_{C(e_1,\sigma_1)}[M_1\cap E_1]$,
\item[(b)]$\lvert g(x+te_1)-g(x)-(f_1(x+te_1)-f_1(x))\rvert\leq\epsilon\lvert t\rvert\leq2^{-10(n+1)}\lvert t\rvert\leq\lvert t\rvert/2$ for any $x\in G\cap (1/k_0,1-1/k_0)^n$ and any $t\in \R$ such that $x+te_1\in[1/2k_0,1-1/2k_0]^n$.
\item[(c)]$f_1$ and $g$ have $2^{-1}$-directional derivative along $e_1$ for any $x\in G$.
\end{itemize}

With these choices of $G$ and $M_1$ points (iv), (v) in the winning strategy are satisfied. Moreover  points (ii), (iii), (vi) and (vii) do not require verification.

\paragraph*{Construction of the $k$-th move of the Player II}
Let $k\geq 2$ and assume:
\begin{equation}
\lip([0,1]^n,\R^m)\supseteq U_1\supseteq V_1\supseteq\ldots\supseteq U_{k-1}\supseteq V_{k-1},\nonumber
\end{equation}
be the match that has been played up to step $k-1$. Moreover assume the functions $f_i$, the closed sets $\mathcal{A}_i$ and non-empty relatively open subsets $M_i$ of $\mathcal{A}_i$ which verify the required conditions have been already defined and suppose $U_k\subseteq V_{k-1}$ is the arbitrary $k$-th move of player I.

The answer of Player II starts with the choice of an arbitrary function $f_k\in U_k\cap \mathfrak{P}(n,m)$ and uses (iv) of the $(k-1)$-step to deduce that there exists an open set $G\subseteq (0,1)^n$ such that:
\begin{itemize}
\item[($\alpha$)] $w_{C(e_{k-1},\sigma_{k-1})}[(1/k_0,1-1/k_0)^n\setminus G]< w_{C(e_{k-1},\sigma_{k-1})}[M_{k-1}\cap E_{k-1}]$.
\item[($\beta$)]$\lvert f_k(x+te_{k-1})-f_k(x)-(f_{k-1}(x+te_{k-1})-f_{k-1}(x))\rvert\leq2^{-k+1}\lvert t\rvert$ for any $x\in G\cap (1/k_0,1-1/k_0)^n$ and any $t\in \R$ such that $x+te_{k-1}\in[1/2k_0,1-1/2k_0]^n$.
\item[($\gamma$)]$f_{k-1}$ and $f_k$ have $2^{-k+1}$-directional derivative along $e_{k-1}$ for any $x\in G$.
\end{itemize}
Let $8\sigma_k:=2^{-100(n+k)}\text{Card}(\Pi_{f_k})^{-2}\sigma_{k-1}$ and assume $\{u_i\}_{i\in\{1,\ldots, N_k\}}$ is a finite $\sigma_k/2$-dense set in $C(e_{k-1},\sigma_{k-1})$. Proposition \ref{spessspess}, point  ($\alpha$) and the fact that $M_{k-1}$ is a relatively open subset of $\mathcal{A}_{k-1}$ imply that the set $G\cap M_{k-1}\cap E_{k-1}$ is also a relatively open non-empty subset of $\mathcal{A}_{k-1}$. 

Therefore  Proposition \ref{splitter} implies that we can find and a closed subset $\mathcal{A}_k$ of $\mathcal{A}_{k-1}$ such that $G\cap M_{k-1}\cap E_{k-1}\cap\mathcal{A}_{k}\neq \emptyset$. In particular this implies that:
\begin{itemize}
    \item $\mathcal{A}_k$ has every portion of positive $C(u_i,\sigma_{k})$-width for some $i\in\{1,\ldots,N_k\}$,
    \item we can find a non-empty relatively open subset $\tilde{M}_k$ of $\mathcal{A}_{k-1}$ such that $\text{cl}(\tilde{M}_k)\subseteq G\cap M_{k-1}\cap E_{k-1}$.
\end{itemize}
We define $e_k:=u_i$ and note that since $E_k$ is a dense relatively open subset of $F$, the set $\tilde{M}_k\cap E_k$ is a non-empty relatively open set in $\mathcal{A}_k$ and thus $w_C(e_k,\sigma_k)[\tilde{M}_k\cap E_k]>0$.
Finally we let $\epsilon_k:=2^{-10(n+k)}(nk_0)^{-1}$ and:
$$\omega_k:=\min\{\epsilon_k,w_{C(e_k,\sigma_k)}(\tilde{M}_k\cap E_k)/(1+n^{1/2})\}.$$

Applying Proposition \ref{lemma1} to $f_k$, $\omega_k$, $e_k$, $\sigma_k$ and $\epsilon_k$ we get an open neighbourhood $V_k$ of $f_k$ (with $\text{diam}(V_k)<2^{-k}$) such that for every $g\in V_k\cap \mathfrak{P}(n,m)$ there exist an open set $\tilde{G}$ for which:
\begin{itemize}
\item[(a')] $w_{C(e_k,\sigma_k)}[(1/k_0,1-1/k_0)^n\setminus \tilde{G}]<w_{C(e_k,\sigma_k)}[\tilde{M}_k\cap E_k]$.
\item[(b')]$\lvert g(x+te_k)-g(x)-(f_k(x+te_k)-f_k(x))\rvert\leq\epsilon_k\lvert t\rvert\leq 2^{-10(n+k)}\lvert t\rvert\leq\lvert t\rvert/2^k$ for any $x\in \tilde{G}\cap (1/k_0,1-1/k_0)^n$ and any $t\in \R$ such that $x+te_k\in[1/2k_0,1-1/2k_0]^n$.
\end{itemize}

Moreover since $\sigma_k<\epsilon^2/2$, applying Proposition \ref{fortuna} to $f_k$, $g$, $e_k$, $\sigma_k$ and $\epsilon_k$ we obtain that $\text{cl}(\Xi(f_k,e_k,\epsilon_k))$ and $\text{cl}(\Xi(g,e_k,\epsilon_k))$ are $C(e_k,\sigma_k)$-null. We define:
\begin{itemize}
\item$M_k:=\tilde{M}_k\setminus\text{cl}(\Xi(f_k,e_k,\epsilon_k))$,
\item $G:=\tilde{G}\setminus(\text{cl}(\Xi(f_k,e_k,\epsilon_k))\cup \text{cl}(\Xi(g,e_k,\epsilon_k)))$,
\end{itemize}
and note that $M_k$ is still relatively open in $\mathcal{A}_k$ and $G$ is open. Thus $G$ is the open set associated to $g$ that satisfies to our requirements:
\begin{itemize}
\item[(a)] $w_{C(e_k,\sigma_k)}[(1/k_0,1-1/k_0)^n\setminus G]< w_{C(e_k,\sigma_k)}[M_k\cap E_k]$,
\item[(b)]$\lvert g(x+te_k)-g(x)-(f_k(x+te_k)-f_k(x))\rvert\leq\epsilon\lvert t\rvert\leq2^{-10(n+k)}\lvert t\rvert\leq\lvert t\rvert/2$ for any $x\in G\cap (1/k_0,1-1/k_0)^n$ and any $t\in \R$ such that $x+te_k\in[1/2k_0,1-1/2k_0]^n$.
\item[(c)]$f_k$ and $g$ have $2^{-k}$-directional derivative along $e_k$ for any $x\in G$.
\end{itemize}
On $M_k$ the function $f_k$ has $2^{-k}$-directional derivative along $e_k$ and since $M_i\subseteq M_k$ for any $i$, we have by induction hypothesis that (v) holds true. 
These choices of $G$ and $M_k$ satisfy points (ii) and (iii) in the winning strategy. Point (iv) holds by construction.

\paragraph*{Proof that the  built strategy is winning}
By construction $\text{cl}(V_{k-1})\subseteq V_k$ and thus, by Arzel\'a-Ascoli theorem and the finite intersection property of compact sets, we deduce that $\bigcap_k V_k\neq \emptyset$. Moreover, since $\text{diam}(V_k)$ converge to $0$, we also deduce that $\bigcap_k V_k=\{f\}\subseteq \lip([0,1]^n,\R^m)$. In the same way, since $\text{cl}(M_k)\subset M_{k-1}\cap E_{k-1}\subseteq F$, we deduce that:
$$\emptyset\neq \bigcap_{k=1}^\infty M_k\subseteq E.$$
Since $e_k\in C(e_{k-1},\sigma_{k-1})$ we have that $\{e_k\}$ converges to some $e\in\mathbb{S}^{n-1}$ and:
$$\lvert e_k-e\rvert\leq 2\sigma_k\leq 2^{-100k}.$$
Fix some $x\in\bigcap_{k=1}^\infty M_k$ and let $d_k$ be the $2^{-k}$-directional derivative along $e_k$ of $f_k$ at $x$. Note that thanks to compactness of the unitary ball of $\R^m$ we can find a sequence $\{k_j\}_{j\in\N}$ and $d\in B_1(0)$ for which $\lvert d_{k_{j}}-d\rvert\leq2^{-j}$. For any $j\geq 1$, since $f$ is $1$-Lipschitz, we have:
\begin{equation}
\begin{split}
&\limsup_{t\to 0}\left\lvert \frac{f(x+te)-f(x)-td}{t}\right\rvert\leq\frac{1}{2^{j-1}}+\limsup_{t\to 0}\left\lvert \frac{f(x+te_{k_j})-f(x)-td_{k_j}}{t}\right\rvert.
\label{numbero2010}
\end{split}
\end{equation}
Using the fact that $f_i$ converges uniformly to $f$ and repetitively applying the triangle inequality, the argument of the limit in the right-hand side of \eqref{numbero2010} becomes:
\begin{equation}
    \begin{split}
    \left\lvert \frac{f(x+te_{k_j})-f(x)-td_{k_j}}{t}\right\rvert\leq& \left\lvert\frac{f_{k_j}(x+te_{k_j})-f_{k_j}(x)-td_{k_j}}{t}\right\rvert\\
    +&\sum_{i=k_j}^\infty \left\lvert\frac{f_{i+1}(x+te_{i+1})-f_{i+1}(x)}{t}-\frac{f_i(x+te_i)-f_i(x)}{t}\right\rvert.
    \label{equ!}
    \end{split}
\end{equation}
Using the fact that $f_i$ are $1$-Lipschitz and the point (vii) of the winning strategy, we deduce that each term in the sum of the right-hand side of the inequality \eqref{equ!} can be estimated by:
\begin{equation}
    \begin{split}
    &\left\lvert\frac{f_{i+1}(x+te_{i+1})-f_{i+1}(x)}{t}-\frac{f_i(x+te_i)-f_i(x)}{t}\right\rvert\\
    \leq & \lvert e_i-e_{i+1}\rvert+\left\lvert\frac{f_{i+1}(x+te_i)-f_{i+1}(x)}{t}-\frac{f_i(x+te_i)-f_i(x)}{t}\right\rvert\leq \frac{1}{2^{i-2}}.
        \nonumber
    \end{split}
\end{equation}
Therefore the bound in \eqref{equ!} becomes:
\begin{equation}
    \begin{split}
    \left\lvert \frac{f(x+te_{k_j})-f(x)-td_{k_j}}{t}\right\rvert\leq \left\lvert\frac{f_{k_j}(x+te_{k_j})-f_{k_j}(x)-td_{k_j}}{t}\right\rvert
    +2^{3-k_j}.
    \nonumber
    \end{split}
\end{equation}
Since by construction $f_{k_j}$ is $2^{-k_j}$-differetiable along $e_{k_j}$ at $x$, the bound \eqref{numbero2010} becomes:
\begin{equation}
\begin{split}
&\limsup_{t\to 0}\left\lvert \frac{f(x+te)-f(x)-td}{t}\right\rvert\leq2^{1-j}+2^{-k_j} +2^{3-k_j},
\nonumber
\end{split}
\end{equation}
which by arbitrariness of $j$ concludes the proof.
\end{proof}

%
%

\end{document}